\documentclass[a4paper]{article}
\usepackage{geometry}
\usepackage{tabularx}
\usepackage{tikz}
\usepackage{overpic}
\usepackage{amsmath, amssymb, amsthm,  amsfonts, mathrsfs}
\usepackage{multirow}
\usepackage{color}
\usepackage{booktabs}
\usepackage{rotating}
\usepackage{soul}
\usepackage{bm}
\usepackage{scrextend}
\usepackage{arydshln}
\usepackage {url}
\usepackage[font=scriptsize,labelfont=bf,justification=centering]{caption}
\usepackage{graphicx} 
\usepackage{epstopdf} 
\usepackage{algpseudocode}
\usepackage{array}
\usepackage{algorithm}
\usepackage{paralist} 
\usepackage{verbatim} 
\usepackage{subfig} 
\usepackage{xcolor, marginnote, enumitem}
\usepackage[numbers]{natbib}

\numberwithin{equation}{section}
\theoremstyle{plain}
\newtheorem{lemma}{Lemma}
\newtheorem{theorem}{Theorem}
\newtheorem{proposition}{Proposition}
\newtheorem{corollary}{Corollary}
\newtheorem{assumption}{Assumption}

\theoremstyle{remark}
\newtheorem{remark}{Remark}

\theoremstyle{definition}
\newtheorem{definition}{Definition}

\DeclareMathOperator{\dist}{dist}

\DeclareMathOperator*{\argmin}{arg\, min}

\newcommand{\lookUp}[1]{}

\newcommand{\bitem}{\begin{itemize}}
	\newcommand{\eitem}{\end{itemize}}

\newcommand{\bpm}{\begin{pmatrix}}

\DeclareMathAlphabet{\mathbfit}{OML}{cmm}{b}{it}

\usepackage{xspace}
\usepackage{bold-extra}
\usepackage[most]{tcolorbox}

\colorlet{texcscolor}{blue!50!black}
\colorlet{texemcolor}{red!70!black}
\colorlet{texpreamble}{red!70!black}
\colorlet{codebackground}{black!25!white!25}

\graphicspath{{./}{pic/}}
\usepackage[plainpages=false, colorlinks=true, 
            linkcolor=black, urlcolor=black, citecolor=black]{hyperref}

\date{}

\begin{document}

\title{A Spectral Proximal DC Algorithm with Nonmonotone Line Search and Extrapolation}

\author{Kelin Wu\thanks{School of Mathematics,
Renmin University of China, China. \href{mailto:kelinwu@ruc.edu.cn}{kelinwu@ruc.edu.cn}}
\and Hongpeng Sun\thanks{School of Mathematics,
Renmin University of China, China. Corresponding author: Hongpeng Sun,
\href{mailto:hpsun@amss.ac.cn}{hpsun@amss.ac.cn}.}
}

\maketitle

\begin{abstract}
The paper proposes a novel proximal difference-of-convex (DC) algorithmic framework to solve general nonconvex, nonsmooth optimization problems. The basic pBBDCA method combines Barzilai-Borwein (BB)
parameter selection with a nonmonotone directional line search, requiring only one proximal solve per iteration. We also develop two extrapolated variants using restart and safeguarding mechanisms. The global convergence of the proposed framework is rigorously established under the Kurdyka-\L ojasiewicz property. Numerical experiments on the SCAD-regularized least squares problem and graph Ginzburg-Landau image segmentation models demonstrate that the proposed methods achieve highly competitive efficiency and accuracy.
\end{abstract}

\paragraph{MSCodes.}{65K05, 65K10, 49J52, 90C26, 90C30}

\section{Introduction}\label{sec:intro}
We focus on the following nonconvex optimization problem
\begin{equation}\label{eq:basefunctional}
    \min_{x \in X} E(x)= f(x) + g_1(x) - g_2(x),
\end{equation}
where $f(x)$ is an $L$-smooth convex function (i.e., $\nabla f(x)$ is Lipschitz continuous with constant $L$). Both $g_1(x)$ and $g_2(x)$ are proper closed convex functions, and $X$ is a finite-dimensional Hilbert space. Nonconvex optimization problems of this form are pervasive in modern computational mathematics, posing significant theoretical challenges \cite{Pangcui}, with critical applications spanning machine learning, signal processing, and imaging \cite{BF, Lethi2021, Li_2020, Shensun2023}. The difference of convex functions algorithm (DCA) is a highly versatile and powerful methodology for tackling this class of problems by leveraging the intrinsic DC structure \cite{LeThi2018, Lethi2024}.

To ensure global convergence and handle the smooth component $f(x)$, the proximal DCA (pDCA) introduces a quadratic proximal term. However, a theoretical limitation of standard pDCA is that it binds the proximal parameter to the global Lipschitz constant $L$. In practical optimization trajectories, the local directional curvature is often significantly smaller than this global worst-case bound. Consequently, standard pDCA cannot adapt to local curvature, forcing overly conservative step sizes and degrading the convergence rate.

To break this fundamental bottleneck, we seek strategies from the Barzilai-Borwein (BB) method \cite{Barzilai88,Birgin2020}. By approximating the inverse Hessian using a scalar matrix $\alpha_n I$, the BB mechanism identifies the active spectral component along the search trajectory \cite{crisci2020spectral}. This allows the algorithm to replace the conservative global bound with an aggressive, curvature-adaptive local step size.

The curvature-adaptive BB parameter does not ensure monotone descent \cite{burdakov2019stabilized}.
We therefore combine spectral parameter selection with
a nonmonotone line search, following related developments
in spectral optimization
\cite{birgin2000nonmonotone,crisci2022convergence,
crisci2024barzilai}.
This combination has also been used for SCAD-type regularization in GIST \cite{gong2013general},
EPDCA \cite{luzhaosong}, and the nonconvex AdProxBB method \cite{pandey2026adaptive}.
These methods backtrack on the proximal parameter and may require repeated proximal evaluations. To avoid these repeated evaluations, our pBBDCA computes a BB-scaled proximal DC point once per iteration and then searches along the resulting direction, keeping the
proximal parameter fixed.

The separation of direction construction from line search is shared
with BDCA \cite{Artacho2018} and its nonmonotone extension
\cite{ferreira2024boosted}.
A related preconditioned DC framework combines extrapolation with line search and establishes KL-based
convergence results \cite{zhang2026preconditioned}.
Whereas these methods search beyond a DCA point, our algorithms search
between the current iterate and a BB-scaled proximal trial point.
This accommodates spectral parameters beyond the global bound $1/L$ without repeated proximal solves during backtracking.

To further transcend the limitations, we integrate odd-even alternating BB step sizes \cite{dai2005projected} with Nesterov-type extrapolation \cite{Lu2019,Wen2018}. However, directly integrating the highly aggressive Barzilai-Borwein method with extrapolation can trigger severe numerical instability. The nonmonotone line search and nonsmoothness also pose a significant theoretical challenge for global convergence.

To bridge this literature gap and formally overcome the theoretical obstacles, our main contributions are articulated as follows:
\begin{itemize}
    \item \textbf{Circumventing the Global Lipschitz Constraint for Proximal DCA:} We formally establish that the BB step sizes act as local directional curvature estimators. By decoupling this aggressive spectral prediction from the nonmonotone line search, we construct the Proximal Barzilai-Borwein DC Algorithm (pBBDCA). Our pBBDCA safely adopts aggressive steps beyond the conservative global Lipschitz bound, while evaluating the expensive proximal subproblem only once per iteration.
    
   \item \textbf{Safeguarded Extrapolation:} To further improve the algorithmic efficiency, we integrate odd-even alternating BB step sizes \cite{dai2005projected} with Nesterov extrapolation \cite{Lu2019,Wen2018}. We propose two advanced variants: pABBDCA$_{\text{er}}$ (with a restart mechanism) and pABBDCA$_{\text{se}}$ (with an adaptively safeguarded extrapolation truncation). Both variants theoretically guarantee global convergence. Furthermore, as demonstrated in our numerical experiments, the integration of extrapolation effectively suppresses the severe numerical oscillations inherent to standard BB methods, yielding a stable and accelerated descent trajectory.
    
    \item \textbf{Novel KL Global Convergence Framework:} We establish a rigorous global convergence theory for the proposed algorithms under the Kurdyka-\L ojasiewicz (KL) property \cite{ABS}. When the convex component \(g_1\) is nonsmooth, subproblem optimality provides subgradient control at the proximal trial point, whereas descent is enforced at a different, accepted iterate. We derive an energy relation linking these points and use it to connect the KL inequality with objective or Lyapunov descent. Under the stated regularity and KL assumptions, this analysis establishes finite-length convergence to a critical point.
\end{itemize}

The remainder of this paper is organized as follows. Section \ref{sec:proposed_method} details the construction of our proposed algorithm, pBBDCA, and its extrapolated variants. Section \ref{sec2} establishes the convergence analysis for pBBDCA. Subsequently, Sections \ref{sec:alg2_conv} and \ref{sec:alg3_conv} prove the convergence of pABBDCA$_{\text{er}}$ and pABBDCA$_{\text{se}}$ respectively. Section \ref{sec:pre} demonstrates the high practical efficiency of our algorithms through numerical experiments on the least-squares problem with SCAD regularization and on graph Ginzburg-Landau models. Finally, Section \ref{sec:conclusion} provides a conclusion.

\section{The Proposed Algorithmic Framework}\label{sec:proposed_method}
 We start by developing a nonmonotone proximal DC algorithm accelerated by Barzilai-Borwein step sizes, and subsequently extend it to incorporate safeguarded extrapolation mechanisms.

\subsection{Proximal Barzilai-Borwein DC Algorithm}

\begin{algorithm}[H] 
\renewcommand{\algorithmicrequire}{\textbf{Input:}}
\renewcommand{\algorithmicensure}{\textbf{Output:}}
\caption{Proximal Barzilai-Borwein DC Algorithm (\textbf{pBBDCA})}\label{alg:prox_bb_dca} 
\begin{algorithmic}[1] 
\State \textbf{Require: }Initial $x^0$, $0 < \alpha_{\min} < \alpha_{\max}$, step size $\alpha_0\in[\alpha_{\min},\alpha_{\max}]$, parameters $\lambda_0=1$, $\rho, \omega \in (0,1)$, $0 < \eta < 1/\alpha_{\max}$. Set $x^{-1} = x^0$ and $n=0$.
\State \textbf{While} stopping criterion is not satisfied \textbf{do}
    \State Compute $\xi^n \in \partial g_2(x^n)$ and solve 
    \begin{equation}\label{eq:algorithmic:prox}
        u^{n} = \arg\min_x \left\{ \langle \nabla f(x^n) - \xi^n, x - x^n \rangle + \frac{1}{2\alpha_n}\|x - x^n\|^2 + g_1(x) \right\}.
    \end{equation}
    \State Define $d^n = u^n - x^n$ and $\nu_n = \frac{\omega}{n+1} \|d^n\|^2$. \textbf{If} $d^n = 0$, STOP and return $x^n$.
    \State \textbf{Line Search:} Find the smallest integer $k \ge 0$ such that $\lambda_n = \lambda_0 \rho^k$ satisfies:
    \begin{equation}\label{eq:algorithmic:nls}
         E(x^n + \lambda_n d^n) \le E(x^n) - \eta \lambda_n \|d^n\|^2 + \nu_n.
    \end{equation}
    \State Update $x^{n+1} = x^n + \lambda_n d^n$. Compute $s^n = x^{n+1} - x^n$ and $y^n = \nabla f(x^{n+1}) - \nabla f(x^n)$.
    \State \textbf{Barzilai-Borwein Step Size Update:} Set $\alpha_{n+1} = \max\left\{\alpha_{\min}, \min\left\{\alpha_{\max}, \frac{\langle s^n, s^n \rangle}{\langle s^n, y^n \rangle}\right\}\right\}$ if $\langle s^n, y^n \rangle > 0$; else $\alpha_{n+1} = \alpha_{\max}$. Set $n = n+1$.
\State \textbf{End While}
\end{algorithmic}
\end{algorithm}

Recall the general nonconvex optimization problem \eqref{eq:basefunctional} where $E(x) = f(x) + g_1(x) - g_2(x)$. By virtue of the proximal point algorithm, the standard proximal DCA (pDCA) generates the next iterate by solving the following subproblem:
\begin{equation*}
      x^{n+1}=\argmin_{x\in X} \left\{\langle\nabla f(x^n)-\xi^n,x\rangle+\frac{L}{2}\|x-x^n\|^2+g_1(x)\right\}, \ \text{where} \ \xi^n\in\partial g_2(x^n).
\end{equation*}

 The rationale for the BB step size $\alpha_n$ stems from quasi-Newton methods, which approximate the inverse Hessian matrix with a scalar matrix $\alpha_n I$. Ideally, this approximation should satisfy
$\alpha_n y^{n-1} = s^{n-1}$,
where $s^{n-1} = x^n - x^{n-1}$ and $y^{n-1} = \nabla f(x^n) - \nabla f(x^{n-1})$. Since a scalar matrix cannot in general satisfy this multidimensional equation, Barzilai and Borwein proposed solving it in a least-squares sense. Minimizing the residual $\|\alpha^{-1} s^{n-1} - y^{n-1}\|^2$ yields the first BB step size (BB1):
$\alpha_n^{\text{BB1}} = \frac{\langle s^{n-1}, s^{n-1} \rangle}{\langle s^{n-1}, y^{n-1} \rangle}$.
Alternatively, minimizing the residual $\|s^{n-1} - \alpha y^{n-1}\|^2$ directly with respect to $\alpha$ yields the second variant (BB2):
$\alpha_n^{\text{BB2}} = \frac{\langle s^{n-1}, y^{n-1} \rangle}{\langle y^{n-1}, y^{n-1} \rangle}$.

Replacing $L$ by $\alpha_n^{-1}$ gives a BB-scaled proximal DC subproblem.
Unlike proximal-parameter backtracking in GIST, EPDCA, and AdProxBB
\cite{gong2013general,luzhaosong,pandey2026adaptive}, pBBDCA fixes $\alpha_n$ and the resulting direction $d^n$
while adjusting only the line-search factor $\lambda_n$.
Thus, each iteration requires only one proximal solve.
The detailed procedure is presented in
Algorithm~\ref{alg:prox_bb_dca}.

The relaxation $\nu_n = \frac{\omega}{n+1} \|d^n\|^2$ permits
temporary objective increases during the early iterations and can reduce backtracking of the trial steps.
The decay of $\frac{\omega}{n+1}$, together with the uniform positive lower bound on $\lambda_n$, ensures eventual monotone sufficient decrease.

\subsection{Alternating Barzilai-Borwein Step Sizes and Extrapolation}

\begin{algorithm}[H]
\renewcommand{\algorithmicrequire}{\textbf{Input:}}
\renewcommand{\algorithmicensure}{\textbf{Output:}}
\caption{Proximal Alternating Barzilai-Borwein Step DC Algorithm with Extrapolation and Restart (\textbf{pABBDCA$_{\text{er}}$})}\label{alg:prox_bb_dca_ext_res}
\begin{algorithmic}[1]
\State \textbf{Require: }Initial $x^0$, $0 < \alpha_{\min} < \alpha_{\max}$, step size $\alpha_0\in[\alpha_{\min},\alpha_{\max}]$, parameters $\lambda_0=1$, $\rho, \omega, \overline{\beta} \in (0,1)$, $0 < \eta < 1/\alpha_{\max}$, max search iterations $N_{\max} \ge 1$. Set $x^{-1} = x^0$, $\beta_0 = 0$, $n=0$.
\State \textbf{While} stopping criterion is not satisfied \textbf{do}
    \State Set $\beta_n = \min\left\{\frac{n-1}{n+2}, \overline{\beta}\right\}(n\ge 1)$.
    \State Set $y^n = x^n + \beta_n (x^n - x^{n-1})$. Compute $\xi^n \in \partial g_2(x^n)$ and solve the proximal subproblem:
    \begin{equation}\label{eq:algorithm2:prox}
        u^{n} = \underset{x}{\arg\min} \left\{ \langle - \xi^n, x - y^n \rangle + \frac{1}{2\alpha_n}\|x - y^n\|^2 + f(x)+ g_1(x) \right\}. 
    \end{equation}
    \State Define $d^n = u^n - x^n$ and $\nu_n = \frac{\omega}{n+1} \|d^n\|^2$.
    \State \textbf{Line Search and Restart:} Find the smallest $k \ge 0$ such that $\lambda_n = \lambda_0 \rho^k$ satisfies $$E(x^n + \lambda_n d^n) \le E(x^n) - \eta \lambda_n \|d^n\|^2 + \nu_n.$$ \textbf{If} $k > N_{\max}$ and $\beta_n > 0$, \textbf{Restart} by setting $\beta_n = 0$, $y^n = x^n$, and \textbf{go to step 4}.
    \State Update $x^{n+1} = x^n + \lambda_n d^n$. Compute $s^n = x^{n+1} - y^n$ and $y_{bb}^n = \nabla f(x^{n+1}) - \nabla f(y^n)$.
      \State \textbf{Alternating Barzilai-Borwein Step Size Update: } 
    \begin{equation*}
         \alpha_{n+1} = \begin{cases} 
         \max\left\{\alpha_{\min}, \min\left\{\alpha_{\max}, \frac{\langle s^n, s^n \rangle}{\langle s^n, y^n_{bb} \rangle}\right\}\right\}, & \text{if } \langle s^n, y^n_{bb} \rangle > 0 \text{ and } n \text{ is even}, \\
         \max\left\{\alpha_{\min}, \min\left\{\alpha_{\max}, \frac{\langle s^n, y^n_{bb} \rangle}{\langle y^n_{bb}, y^n_{bb} \rangle}\right\}\right\}, & \text{if } \langle s^n, y^n_{bb} \rangle > 0 \text{ and } n \text{ is odd}, \\
         \alpha_{\max}, & \text{otherwise}.
         \end{cases}
    \end{equation*} 
    Set $n = n+1$.
\State \textbf{End While}
\end{algorithmic}
\end{algorithm}

To further improve computational efficiency, we integrate Nesterov-type extrapolation into our framework. We define the extrapolated point $y^n$ as
\begin{equation*}
y^n = x^n + \beta_n (x^n - x^{n-1}),
\end{equation*}
where $\beta_n \in [0, 1)$ dictates the extrapolation weight. The proximal subproblem is subsequently evaluated at $y^n$ rather than $x^n$.

However, in nonconvex and nonsmooth environments, the combination of extrapolation and nonmonotone line search requires careful geometric stabilization. Standard BB step sizes might occasionally generate trial directions that deviate significantly from the steepest descent, triggering excessive backtracking when extrapolation is applied. To counteract this, we employ an Alternating Barzilai-Borwein (ABB) strategy \cite{dai2005projected}. By alternating between the longer BB1 steps (even iterations) and more conservative BB2 steps (odd iterations), the ABB rule naturally smooths the curvature estimation.

\begin{algorithm}[H]
\renewcommand{\algorithmicrequire}{\textbf{Input:}}
\renewcommand{\algorithmicensure}{\textbf{Output:}}
\caption{Proximal Alternating Barzilai-Borwein Step DC Algorithm with Safeguarded Extrapolation (\textbf{pABBDCA$_{\text{se}}$})}\label{alg:prox_bb_dca_ext_ada}
\begin{algorithmic}[1]
\State \textbf{Require: }Initial $x^0$, $0 < \alpha_{\min} < \alpha_{\max}$, step size $\alpha_0\in[\alpha_{\min},\alpha_{\max}]$, $\lambda_0 =1, \rho, \omega, \overline{\beta} \in (0,1)$, $0 < \eta < 1/\alpha_{\max}$, $c_0 > 0$. Set $x^{-1} = x^0$, $\beta_0 = 0$, $\lambda_{-1} = 1$, $\alpha_{-1} = \alpha_0$, $n=0$.
\State \textbf{While} stopping criterion is not satisfied \textbf{do}
    \State \textbf{Safeguarded Extrapolation: }If $n \ge 1$, set $\hat{\beta}_n = \min\left\{\frac{n-1}{n+2}, \overline{\beta}\right\}$ and
    \begin{equation}\label{betan}
        \beta_n = \max\left\{0, \min\left\{\hat{\beta}_n, \frac{\alpha_n}{\alpha_{n-1}\lambda_{n-1}}(2-\beta_{n-1}) - 2c_0\alpha_n\right\}\right\}.
    \end{equation}
    \State Set $y^n = x^n + \beta_n (x^n - x^{n-1})$. Compute $\xi^n \in \partial g_2(x^n)$ and solve the proximal subproblem:
    $$ u^{n} = \underset{x}{\arg\min} \left\{ \langle - \xi^n, x - y^n \rangle + \frac{1}{2\alpha_n}\|x - y^n\|^2 + f(x)+ g_1(x) \right\}. $$
    \State Define $d^n = u^n - x^n$ and $\nu_n = \frac{\omega}{n+1} \|d^n\|^2$.
    \State \textbf{Line Search: }Find the smallest integer $k \ge 0$ such that $\lambda_n = \lambda_0 \rho^k$ satisfies $$E(x^n + \lambda_n d^n) \le E(x^n) - \eta \lambda_n \|d^n\|^2 + \nu_n.$$ Update $x^{n+1} = x^n + \lambda_n d^n$. Compute $s^n = x^{n+1} - y^n$ and $y^n_{bb} = \nabla f(x^{n+1}) - \nabla f(y^n)$.
    \State \textbf{Alternating Barzilai-Borwein Step Size Update: } 
    \begin{equation*}
         \alpha_{n+1} = \begin{cases} 
         \max\left\{\alpha_{\min}, \min\left\{\alpha_{\max}, \frac{\langle s^n, s^n \rangle}{\langle s^n, y^n_{bb} \rangle}\right\}\right\}, & \text{if } \langle s^n, y^n_{bb} \rangle > 0 \text{ and } n \text{ is even}, \\
         \max\left\{\alpha_{\min}, \min\left\{\alpha_{\max}, \frac{\langle s^n, y^n_{bb} \rangle}{\langle y^n_{bb}, y^n_{bb} \rangle}\right\}\right\}, & \text{if } \langle s^n, y^n_{bb} \rangle > 0 \text{ and } n \text{ is odd}, \\
         \alpha_{\max}, & \text{otherwise}.
         \end{cases}
    \end{equation*}
    \quad Set $n = n+1$.
\State \textbf{End While}
\end{algorithmic}
\end{algorithm}

While extrapolation accelerates convergence, unconstrained extrapolated steps in a nonmonotone setting can compromise overall algorithmic stability. To harness the acceleration benefits without sacrificing theoretical robustness, we introduce two distinct safeguarded extrapolation mechanisms.

The first approach is a heuristic \textbf{Restart Mechanism} (Algorithm \ref{alg:prox_bb_dca_ext_res}). In this case, if the line search fails to find a suitable step size
within a predefined maximum number of inner iterations
($N_{\max}$), the algorithm resets the extrapolation weight
$\beta_n=0$ and re-evaluates the proximal subproblem at $x^n$.

For the second approach, we propose a \textbf{Safeguarded Extrapolation Mechanism} (Algorithm \ref{alg:prox_bb_dca_ext_ada}). This method adaptively computes a theoretically safe upper bound for the extrapolation parameter $\beta_n$ based on historical step sizes and gradient curvatures. By enforcing this bound, the algorithm guarantees global convergence, thereby preserving both theoretical robustness and computational efficiency.

 To ensure our algorithmic framework covers a broad class of applications, we categorize our analysis into two distinct regularity scenarios, formalized as follows.

\begin{assumption}\label{assum:smooth_g1}
    The function $f$ is continuously differentiable with an $L$-Lipschitz continuous gradient. The component $g_1$ is continuously differentiable with a locally Lipschitz continuous gradient, and $g_2$ is a proper, lower semicontinuous convex function.
\end{assumption}

\begin{assumption}\label{assum:smooth_g2}
    The function $f$ is continuously differentiable with an $L$-Lipschitz continuous gradient. The component $g_2$ is continuously differentiable with a locally Lipschitz continuous gradient, and $g_1$ is a proper, lower semicontinuous convex function.
\end{assumption}

To formally capture the curvature adaptation, we present the definition of local Lipschitz continuity.

\begin{definition}[Local Lipschitz Continuity]\label{def:local_lip}
(\cite[Definition 9.1]{rockafellar1998variational}) A function $h: \mathbb{R}^d \to \mathbb{R}$ is said to be locally Lipschitz continuous if its Lipschitz modulus is finite at every $x \in \mathbb{R}^d$:
\begin{equation}
    \text{lip} \, h(x) := \limsup_{\substack{y, z \to x \\ y \neq z}} \frac{\|h(y) - h(z)\|}{\|y - z\|} < +\infty.
\end{equation}
By the definition of the limit supremum, for any point $x$, there inherently exists an open neighborhood $U_x$ and a local constant $L_x > 0$ such that $\|h(y) - h(z)\| \le L_x \|y - z\|$ for all $y, z \in U_x$.
\end{definition}

For each iteration, let
\[
    \Gamma_n:=\text{conv}\{x^n,u^n\}
    =\{x^n+t(u^n-x^n):t\in[0,1]\}.
\]
For each component $h$ whose gradient is locally Lipschitz
continuous under the applicable assumption, define the
segmentwise Lipschitz modulus
\[
    \ell_{h,n}
    :=
    \sup_{\substack{z,w\in\Gamma_n\\z\neq w}}
    \frac{\|\nabla h(z)-\nabla h(w)\|}{\|z-w\|}.
\]
When $u^n=x^n$, we set $\ell_{h,n}=0$.
Local Lipschitz continuity of $\nabla h$ on a neighborhood of
the compact segment $\Gamma_n$ ensures that $\ell_{h,n}$ is finite.
In particular,
\[
    \|\nabla h(z)-\nabla h(w)\|
    \leq \ell_{h,n}\|z-w\|,
    \qquad z,w\in\Gamma_n.
\]
We use this notation for $h=f$, for $h=g_1$ under
Assumption~\ref{assum:smooth_g1}, and for $h=g_2$ under
Assumption~\ref{assum:smooth_g2}.

\section{Convergence analysis of pBBDCA}\label{sec2}
For the convergence analysis, we need the following  Kurdyka-\L ojasiewicz (KL) property and KL exponent.  
\begin{definition}[KL Property, KL Function, and KL Exponent]\label{def:KL}
(\cite[Definition 2.4]{ABS}, \cite[Definition 1]{Artacho2018}, and \cite[Remark 6]{Bolte2014}) Let $h:\mathbb{R}^d\to(-\infty,+\infty]$ be a proper lower semicontinuous function. $h$ is said to satisfy the KL property if for any critical point $\bar x$,  there exists $\nu \in (0,+\infty]$, a neighborhood $\mathcal{O}$ of $\bar x$, and a continuous concave function $\psi: [0,\nu) \rightarrow [0,+\infty)$ with $\psi(0)=0$ such that:
	\begin{itemize}
		\item [{\rm{(i)}}] $\psi$ is continuously differentiable on $(0,\nu)$ with $\psi'>0$ over $(0,\nu)$;
		\item [{\rm{(ii)}}] for any $x \in \mathcal{O}$ with $h(\bar x) < h(x) <h(\bar x) + \nu$, one has
		\begin{equation}\label{eq:kl:def}
		\psi'(h(x)-h(\bar x)) \cdot\text{dist}(\textbf{0},\partial h(x))\geq 1 . 
		\end{equation}
	\end{itemize}
	Furthermore, for a proper closed function $h$ satisfying the KL property,  if $\psi$ in \eqref{eq:kl:def} can be chosen as $\psi(s) = cs^{1-\theta}$ for some $\theta \in [0,1)$ and $c>0$, i.e., there exist  $\bar c, \epsilon >0$ such that
	\begin{equation}\label{eq:KL:exponent:theta:exam}
	\text{dist}(\textbf{0},\partial h(x)) \geq \bar c (h(x)-h(\bar x))^{\theta}
	\end{equation}
 whenever $\|x -\bar x\| \leq \epsilon$ and $h(\bar x) < h(x) <h(\bar x) + \nu$,  we say that $h$ has the KL property at $\bar x$ with exponent $\theta$.
\end{definition}
To establish that $h$ is a KL function with exponent $\theta$,
it suffices to verify this property at every critical point
\cite[Lemma 2.1]{Li2018}. It is also assumed that the energy $E(x)$ is bounded from below, and for every scalar $\alpha \in \mathbb{R}$, $\text{lev}_{\leq \alpha}(E): = \{x: E(x) \leq \alpha\}$ is compact.

The following lemma establishes an upper bound on the variation of the objective function along the search direction $d^n = u^n - x^n$, which provides a sufficient condition for objective decrease.

\begin{lemma}\label{lemma:sufficient_descent}
 Let $u^n$ be the optimal solution of the proximal subproblem \eqref{eq:algorithmic:prox} and $d^n = u^n - x^n$ be the search direction. Then, for any step size $\lambda_n \in (0, 1]$, the following descent inequality holds:
\begin{equation}\label{eq:descent_lemma}
    E(x^n + \lambda_n d^n) \le E(x^n) - \lambda_n \left( \frac{1}{\alpha_n} - \frac{L \lambda_n}{2} \right) \|d^n\|^2.
\end{equation}
\end{lemma}

\begin{proof}
Let $Q_n(x)$ denote the objective function of the subproblem \eqref{eq:algorithmic:prox}, which is given by
\begin{equation}
    Q_n(x) = \langle \nabla f(x^n) - \xi^n, x - x^n \rangle + \frac{1}{2\alpha_n} \|x - x^n\|^2 + g_1(x),
\end{equation}
where $\xi^n \in \partial g_2(x^n)$. Since $Q_n(x)$ is strongly convex with modulus $\frac{1}{\alpha_n}$ and 
   $u^n = \arg\min\limits_{x \in \mathbb{R}^d} Q_n(x)$,
we have
\begin{equation}\label{eq:strong_convex_Q}
    Q_n(x^n) \ge Q_n(u^n) + \frac{1}{2\alpha_n} \|x^n - u^n\|^2.
\end{equation}
Due to  $Q_n(x^n) = g_1(x^n)$, expanding $Q_n(u^n)$ and substituting it back into \eqref{eq:strong_convex_Q} yields
\begin{equation}
    g_1(x^n) \ge \langle \nabla f(x^n) - \xi^n, u^n - x^n \rangle + \frac{1}{2\alpha_n} \|u^n - x^n\|^2 + g_1(u^n) + \frac{1}{2\alpha_n} \|x^n - u^n\|^2.
\end{equation}
Noting that $d^n = u^n - x^n$, we obtain
\begin{equation}\label{eq:g1_variation}
    g_1(u^n) - g_1(x^n) \le -\langle \nabla f(x^n) - \xi^n, d^n \rangle - \frac{1}{\alpha_n} \|d^n\|^2.
\end{equation}
Since $f$ is $L$-smooth, we get
\begin{equation}\label{eq:f_bound}
    f(x^n + \lambda_n d^n) \le f(x^n) + \lambda_n \langle \nabla f(x^n), d^n \rangle + \frac{L \lambda_n^2}{2} \|d^n\|^2.
\end{equation}
With the convexity of $g_1$, Jensen's inequality provides
\begin{equation}\label{eq:g1_bound}
    \begin{aligned}
    g_1(x^n + \lambda_n d^n) &= g_1((1-\lambda_n)x^n + \lambda_n u^n)\\ &\le (1-\lambda_n)g_1(x^n) + \lambda_n g_1(u^n) = g_1(x^n) + \lambda_n(g_1(u^n) - g_1(x^n)).
\end{aligned}
\end{equation}
For the concave component $-g_2$, since $\xi^n \in \partial g_2(x^n)$, we have
\begin{equation}\label{eq:g2_bound}
g_2(x^n + \lambda_n d^n) \ge g_2(x^n) + \lambda_n \langle \xi^n, d^n \rangle.
\end{equation}
Combining inequalities \eqref{eq:f_bound}, \eqref{eq:g1_bound} and \eqref{eq:g2_bound}, the upper bound for $E(x^n + \lambda_n d^n)$ becomes
\begin{equation}
    E(x^n + \lambda_n d^n) \le E(x^n) + \lambda_n \langle \nabla f(x^n) - \xi^n, d^n \rangle + \frac{L \lambda_n^2}{2} \|d^n\|^2 + \lambda_n(g_1(u^n) - g_1(x^n)).
\end{equation}
Substituting \eqref{eq:g1_variation} into the above inequality, we obtain
\begin{align*}
    E(x^n + \lambda_n d^n) \le &E(x^n) +\lambda_n \langle \nabla f(x^n) - \xi^n, d^n \rangle + \frac{L \lambda_n^2}{2} \|d^n\|^2 +\lambda_n \left( -\langle \nabla f(x^n) - \xi^n, d^n \rangle - \frac{1}{\alpha_n} \|d^n\|^2 \right) \\
    = &E(x^n) - \lambda_n \left( \frac{1}{\alpha_n} - \frac{L \lambda_n}{2} \right) \|d^n\|^2. 
\end{align*}
\end{proof}

The following lemma demonstrates that the nonmonotone line search is well-defined and guarantees a positive lower bound for the accepted step size.

\begin{lemma}\label{lem:well-definedness}
Assume that $0 < \eta < \frac{1}{\alpha_{\max}}$. Then, the nonmonotone line search in Algorithm \ref{alg:prox_bb_dca} will terminate in a finite number of steps. Moreover, the accepted step size $\lambda_n$ is uniformly bounded away from zero, i.e., there exists a constant $\lambda_{\min} > 0$ such that $\lambda_n \ge \lambda_{\min}$ for all $n \ge 0$.
\end{lemma}

\begin{proof}
From Lemma \ref{lemma:sufficient_descent}, the objective function satisfies
\begin{equation}\label{eq:proof_descent_recall}
    E(x^n + \lambda_n d^n) \le E(x^n) - \lambda_n\left(\frac{1}{\alpha_n} - \frac{L\lambda_n}{2}\right)\|d^n\|^2.
\end{equation}
According to Algorithm \ref{alg:prox_bb_dca}, the nonmonotone line search condition is given by
\begin{equation}
    E(x^n + \lambda_n d^n) \le E(x^n) - \eta \lambda_n\|d^n\|^2 + \nu_n.
\end{equation}
Since $\nu_n = \frac{\omega}{n+1}\|d^n\|^2 \ge 0$, it is sufficient to find a step size $\lambda_n$ that satisfies
\begin{equation}
    E(x^n) - \lambda_n\left(\frac{1}{\alpha_n} - \frac{L\lambda_n}{2}\right)\|d^n\|^2 \le E(x^n) - \eta \lambda_n \|d^n\|^2.
\end{equation}
As $\alpha_n$ is uniformly bounded by $\alpha_n \le \alpha_{\max}$, the sufficient condition reduces to
\begin{equation}
    \lambda_n \le \frac{2}{L}\left(\frac{1}{\alpha_{\max}} - \eta\right) := \bar{\lambda}.
\end{equation}
    Due to $\eta < \frac{1}{\alpha_{\max}}$, it is guaranteed that $\bar{\lambda} > 0$. Consequently, the backtracking loop will terminate  within finite iterations, and the final accepted step size $\lambda_n$ is bounded below by
\begin{equation}
    \lambda_n \ge \lambda_{\min} := \min\left\{\lambda_0, \rho \bar{\lambda}\right\} > 0.
\end{equation}
This confirms that the line search is well-defined and completes the proof. 
\end{proof}

We next establish boundedness of the iterates and the proximal trial points.

\begin{proposition}\label{prop:boundedness}
Let $\{x^n\}_n$ and $\{u^n\}_n$ be the sequences generated by Algorithm \ref{alg:prox_bb_dca}. There exists a compact set $\Omega = \overline{B}(\mathbf{0}, R_{\max})$ such that $x^n, u^n \in \Omega$ for all $n \ge 0$. Consequently, there exists a uniform constant $L_{\Omega} > 0$ such that $\nabla f$ is $L_{\Omega}$-Lipschitz continuous over $\Omega$.
\end{proposition}

\begin{proof}
By the design of the nonmonotone line search, every accepted iterate satisfies
\begin{equation}
    E(x^{n+1}) \le E(x^n) - \eta \lambda_n \|d^n\|^2 + \nu_n = E(x^n) - \left( \eta \lambda_n - \frac{\omega}{n+1} \right) \|d^n\|^2.
\end{equation}
From Lemma \ref{lem:well-definedness}, the step size is universally bounded below by $\lambda_n \ge \lambda_{\min} > 0$. Since $\lim_{n \to \infty} \frac{\omega}{n+1} = 0$, there exists a finite integer $N_0 > 0$ such that for all $n \ge N_0$, we have $\frac{\omega}{n+1} \le \frac{\eta \lambda_{\min}}{2}$. Consequently, the coefficient satisfies $\eta \lambda_n - \frac{\omega}{n+1} \ge \frac{\eta \lambda_{\min}}{2} > 0$, ensuring the eventual monotonic decrease $E(x^{n+1}) \le E(x^n)$ for all $n \ge N_0$. Defining $E_{\max} := \max\left(\{E(x^k)\}_{k=0}^{N_0}\right) < \infty$, we directly conclude $E(x^n) \le E_{\max}$ for all $n \ge 0$.

Since $E(x)$ is level-bounded, $\Omega_1 = \text{lev}_{\le E_{\max}}(E)$ is compact. Thus, the main sequence is trapped inside $\Omega_1$, meaning there exists $R_1 > 0$ such that $\|x^n\| \le R_1$. Noting
$u^n=x^n+\lambda_n^{-1}(x^{n+1}-x^n)$
and $\lambda_n\ge\lambda_{\min}>0$, we have
\[
    \|u^n\|
    \le R_1+\frac{2R_1}{\lambda_{\min}}
    =:R_2.
\]
Thus, $\{u^n\}_n$ is bounded.

Let $R_{\max}=\max\{R_1,R_2\}$.
We construct the closed ball
$\Omega=\overline{B}(\mathbf{0},R_{\max})$,
a convex and compact set containing all iterates.
By local Lipschitz continuity of $\nabla f$,
every point $x\in\Omega$ admits an open neighborhood
$U_x$ and a constant $L_x>0$ such that
\[
    \|\nabla f(z)-\nabla f(w)\|
    \le L_x\|z-w\|,
    \qquad z,w\in U_x.
\]
By the finite subcover theorem applied to
$\{U_x\}_{x\in\Omega}$, there exist
$U_{x_1},\ldots,U_{x_m}$ covering $\Omega$.
Define
\[
    L_{\Omega}:=\max_{1\le i\le m}L_{x_i}.
\]
By the Lebesgue number lemma, there exists $\delta>0$
such that every subset of $\Omega$ with diameter
less than $\delta$ is contained in one of these
neighborhoods.

For any distinct $z,w\in\Omega$, convexity ensures
that $[z,w]\subset\Omega$.
Partition this segment into points
$z=z_0,z_1,\ldots,z_N=w$ so that each subsegment
has length less than $\delta$.
Each subsegment is contained in some $U_{x_i}$.
Applying the corresponding local estimate and
summing along the segment gives
\begin{align*}
    \|\nabla f(z)-\nabla f(w)\|
    \le
    \sum_{j=1}^{N}
    \|\nabla f(z_j)-\nabla f(z_{j-1})\|\le
    L_{\Omega}\sum_{j=1}^{N}\|z_j-z_{j-1}\|=L_{\Omega}\|z-w\|.
\end{align*}
Thus, $\nabla f$ is $L_{\Omega}$-Lipschitz continuous
on $\Omega$.
The same argument yields constants
$L_{g_1,\Omega}$ and $L_{g_2,\Omega}$
under the respective regularity assumptions.
\end{proof}

Since $\Omega$ is convex and contains $x^n$ and $u^n$,
we have $\Gamma_n\subset\Omega$.
Consequently,
\[
    \ell_{f,n}\leq L_{\Omega},
    \qquad
    \ell_{g_1,n}\leq L_{g_1,\Omega},
    \qquad
    \ell_{g_2,n}\leq L_{g_2,\Omega},
\]
where the latter two inequalities are used under the
respective smoothness assumptions.

By Proposition \ref{prop:boundedness}, both the current iterate $x^n$ and the subproblem minimizer $u^n$ are confined within the convex compact set $\Omega$. Since $\Omega$ is convex, any trial point generated by the line search along the segment $x^n + \lambda d^n$ inherently belongs to $\Omega$. The following estimate retains the Lipschitz modulus of
$\nabla f$ on the current segment $\Gamma_n$.

\begin{lemma}\label{lemma:local_descent}
For any accepted step size $\lambda_n \in (0, 1]$, the variation of the objective function is governed by the local Lipschitz constant, satisfying
\begin{equation}\label{eq:descent_lemma_local}
    E(x^{n+1}) \le E(x^n) - \lambda_n \left( \frac{1}{\alpha_n} - \frac{\ell_{f,n} \lambda_n}{2} \right) \|d^n\|^2.
\end{equation}
\end{lemma}

\begin{proof}
Following the identical derivations as in Lemma \ref{lemma:sufficient_descent}, but replacing the global Taylor expansion bound \eqref{eq:f_bound} with the local variation $f(x^{n+1}) \le f(x^n) + \lambda_n \langle \nabla f(x^n), d^n \rangle + \frac{\ell_{f,n} \lambda_n^2}{2} \|d^n\|^2$, the inequality \eqref{eq:descent_lemma_local} directly follows. 
\end{proof}

To demonstrate geometrically how the BB step sizes capture
local curvature, we use the gradient differences along the
search trajectory. Let
\[
    s^{n-1}=x^n-x^{n-1},
    \qquad
    y^{n-1}=\nabla f(x^n)-\nabla f(x^{n-1}).
\]
For $s^{n-1}\ne0$, we define the local directional
curvature of the smooth component $f$ by
\[
    \kappa_{n-1}
    :=\frac{\langle s^{n-1},y^{n-1}\rangle}
            {\langle s^{n-1},s^{n-1}\rangle}.
\]

\begin{proposition}\label{prop:spectral_curvature}
Suppose the local directional curvature satisfies
$\kappa_{n-1}>0$.
The BB1 parameter equals the inverse of this curvature,
i.e.,
\[
    \alpha_n^{BB1}
    =\frac{\langle s^{n-1},s^{n-1}\rangle}
           {\langle s^{n-1},y^{n-1}\rangle}
    =\frac{1}{\kappa_{n-1}}.
\]
For the BB1 choice $\alpha_n=\alpha_n^{BB1}$,
define the effective scalar step size as
$\tau_n:=\lambda_n\alpha_n$.
Its ratio to the standard pDCA step size
$\tau_{pDCA}=1/L$ is
\begin{equation}\label{eq:step_ratio}
    \frac{\tau_n}{\tau_{pDCA}}
    =\lambda_n\frac{L}{\kappa_{n-1}}.
\end{equation}
\end{proposition}

\begin{remark}
The above identities follow directly from the BB1 definition and are included only to illustrate its curvature adaptation.
When $\kappa_{n-1}\ll L$, the BB1 trial parameter can be considerably larger than $1/L$, but the accepted step also depends on $\lambda_n$. 
\end{remark}

\begin{lemma}\label{lem:summability}
Let $\{x^n\}_n$ be the sequence generated by Algorithm \ref{alg:prox_bb_dca}. For all $n \ge N_0$, the sequence of energy values $\{E(x^n)\}_{n \ge N_0}$ is monotonically decreasing. Furthermore, we have
\begin{equation}\label{squaresum}
     \sum_{n=0}^{\infty} \|x^{n+1} - x^n\|^2 < \infty,
\end{equation}
which immediately implies $\lim_{n \to \infty} \|x^{n+1} - x^n\| = 0$ and $\lim_{n \to \infty} \|u^n - x^n\| = 0$.
\end{lemma}

\begin{proof}
Recall from Proposition \ref{prop:boundedness} that for all $n \ge N_0$, the coefficient in the descent inequality satisfies $\eta \lambda_n - \frac{\omega}{n+1} \ge \frac{\eta \lambda_{\min}}{2} := M > 0$. Thus, we have
\begin{equation}\label{eq:eventual_descent}
    E(x^{n+1}) \le E(x^n) - M \|d^n\|^2, \quad \forall n \ge N_0.
\end{equation}
This demonstrates that $\{E(x^n)\}_{n \ge N_0}$ is monotonically decreasing. 
Since $E(x)$ is bounded from below by assumption, summing \eqref{eq:eventual_descent} from $n = N_0$ to $\infty$ yields
\begin{equation}\label{eq:d_summable}
    M \sum_{n=N_0}^{\infty} \|d^n\|^2 \le E(x^{N_0}) - \inf_x E(x) < \infty,
\end{equation}
which implies $\sum_{n=0}^{\infty} \|d^n\|^2 < \infty$.
Finally, noting that $x^{n+1} - x^n = \lambda_n d^n$ and the step size is bounded by $\lambda_n \le \lambda_0$, we obtain
\begin{equation}
    \sum_{n=0}^{\infty} \|x^{n+1} - x^n\|^2 \le \lambda_0^2 \sum_{n=0}^{\infty} \|d^n\|^2 < \infty.
\end{equation}
Consequently, $\|d^n\| \to 0$, which directly implies $\|u^n - x^n\| \to 0$. 
\end{proof}

\subsection{Global convergence under Assumption \ref{assum:smooth_g1}}
We now establish global convergence using a conjugate-based
auxiliary function, as considered in
\cite[Sections~3 and~6]{Liu2019}.

\begin{theorem}\label{global}
Let $\{x^{n}\}_{n}$ be generated by Algorithm \ref{alg:prox_bb_dca} for solving \eqref{eq:basefunctional}. Under Assumption \ref{assum:smooth_g1}, define the auxiliary function
\begin{equation}\label{eq:conj_potential}
    \mathcal{E}(x, \xi) := f(x) + g_1(x) - \langle x, \xi \rangle + g_2^*(\xi),
\end{equation}
where $g_2^*$ is the Fenchel conjugate of $g_2$. If $\mathcal{E}(x, \xi)$ is a KL function, then the following properties hold:
\begin{enumerate}[label=\textup{(\roman*)}]
    \item \textbf{Monotonicity:} The sequence $\{E(x^n)\}_n$ is monotonically decreasing for any $n>N_0$ and there exists a constant $\zeta$, such that $\lim\limits_{n\rightarrow\infty}E(x^{n})=\zeta$.
    \item \textbf{Boundedness:} The iterates $\{x^{n}\}_{n}$ and $\{\xi^{n}\}_{n}$ are bounded.
    \item \textbf{Optimality:} Any cluster point of $\{x^{n}\}_{n}$ is a critical point of the problem \eqref{eq:basefunctional}.
    \item \textbf{Convergence:} The sequence $\{x^{n}\}_{n}$ is globally convergent with $\sum_{n=1}^{\infty}\|x^{n}-x^{n-1}\|<\infty$.
\end{enumerate}
\end{theorem}

\begin{proof}
(i) From \eqref{eq:eventual_descent}, we have that for $n > N_0$, $E(x^{n+1}) \le E(x^n) - \frac{\eta \lambda_{\min}}{2} \|d^n\|^2$.
This indicates that $\{E(x^n)\}_{n>N_0}$ is monotonically decreasing. With the lower boundedness of $E(x)$, there exists a constant $\zeta$, such that $\lim_{n\to\infty} E(x^n)=\zeta$.

    (ii) The monotonic decrease of $\{E(x^n)\}_{n>N_0}$ implies $E(x^n) \le E(x^{N_0}), \forall n>N_0$. This indicates that $\{x^n\}_{n>N_0}$ is entirely contained in the lower level set $\text{lev}_{\leq E(x^{N_0})}(E)$. Since $E(x)$ is assumed to be level-bounded, this set is bounded. Consequently, with the finiteness of the initial $N_0$ terms, the entire sequence $\{x^n\}_n$ is bounded. Moreover, by Proposition~\ref{prop:boundedness},
$\{u^n\}_n$ is bounded. Under
Assumption~\ref{assum:smooth_g1},
the optimality condition of
\eqref{eq:algorithmic:prox} gives
\[
\xi^n
=\nabla f(x^n)+\nabla g_1(u^n)
+\frac{1}{\alpha_n}(u^n-x^n).
\]
Since $\{x^n\}_n$ and $\{u^n\}_n$
are bounded and
$\alpha_n\ge\alpha_{\min}>0$,
the sequence $\{\xi^n\}_n$
is also bounded.

(iii) Since we have proved the boundedness of $\{x^n\}_n$, there exists at least one cluster point $x^*$ with a convergent subsequence $\{x^{n_i}\}_i$ satisfying 
\begin{align}\label{ni}
    \lim_{i\to\infty} x^{n_i}=x^* \quad \text{and}\quad
    \lim_{i\to\infty} \|x^{n_i+1}-x^{n_i}\|=0.
\end{align}
From the first-order optimality condition of the proximal update \eqref{eq:algorithmic:prox}, we get
\begin{align*}
    \nabla f(x^{n_i})+\frac{1}{\alpha_{n_i}}(u^{n_i}-x^{n_i})+\nabla g_1(u^{n_i}) = \xi^{n_i} \in \partial g_2(x^{n_i}).
\end{align*}
Thus we have
\begin{align*}
   \frac{1}{\alpha_{n_i}}(x^{n_i}-u^{n_i}) +[\nabla g_1(x^{n_i})-\nabla g_1(u^{n_i})]\in  \nabla f(x^{n_i})+\nabla g_1(x^{n_i})-\partial g_2(x^{n_i}).
\end{align*}
Since $\ell_{g_1,n_i}\leq L_{g_1,\Omega}$ and $\|u^{n_i}-x^{n_i}\|\to0$ by \eqref{squaresum}, we get
\begin{align*}
   \dist(\textbf{0}, \nabla f(x^{n_i})+\nabla g_1(x^{n_i})-\partial g_2(x^{n_i}))\rightarrow 0.
\end{align*}
With \eqref{ni} and the closedness of $\partial g_2$, we conclude that 
\begin{align*}
    \textbf{0} \in  \nabla f(x^*)+\nabla g_1(x^*)-\partial g_2(x^*),
\end{align*}
which means that $x^*$ satisfies the first-order optimality condition.

(iv) Since $\xi^n \in \partial g_2(x^n)$, the Fenchel-Young equality yields $g_2(x^n) = \langle x^n, \xi^n \rangle - g_2^*(\xi^n)$ and $x^n \in \partial g_2^*(\xi^n)$. Thus, we have $\mathcal{E}(x^n, \xi^n) = E(x^n)$. Under Assumption \ref{assum:smooth_g1}, both $f$ and $g_1$ are continuously differentiable. Therefore, the limiting subdifferential of $\mathcal{E}$ at $(x^n, \xi^n)$ is given by
\begin{align*}
    \partial \mathcal{E}(x^n, \xi^n) = \begin{pmatrix}
        \nabla f(x^n) + \nabla g_1(x^n) - \xi^n \\
        -x^n + \partial g_2^*(\xi^n)
    \end{pmatrix}.
\end{align*}
From the optimality condition of the subproblem \eqref{eq:algorithmic:prox}, we have
\begin{align*}
    \nabla f(x^n) - \xi^n = -\nabla g_1(u^n) - \frac{1}{\alpha_n}(u^n - x^n).
\end{align*}
Substituting this relation into $\partial \mathcal{E}(x^n, \xi^n)$ and noting that $\textbf{0} = -x^n + x^n \in -x^n + \partial g_2^*(\xi^n)$, we obtain
\begin{align*}
    \begin{pmatrix}
        \nabla g_1(x^n) - \nabla g_1(u^n) - \frac{1}{\alpha_n}(u^n - x^n) \\
        \textbf{0}
    \end{pmatrix} \in \partial \mathcal{E}(x^n, \xi^n).
\end{align*}
Consequently, we can bound the distance from the origin to $\partial \mathcal{E}(x^n, \xi^n)$ as
\begin{align*}
  \dist(\textbf{0},\partial \mathcal{E}(x^n, \xi^n)) &\le \|\nabla g_1(x^n) -\nabla g_1(u^n)\| + \left\|\frac{1}{\alpha_n}(u^n-x^n) \right\|\\
    &\le \left(\ell_{g_1,n}+\frac{1}{\alpha_n}\right)\|u^n-x^n\|
    \le C\|u^n-x^n\|,
\end{align*}
where $C = L_{g_1, \Omega} + \frac{1}{\alpha_{\min}}$ is a positive constant. With $\lim\limits_{n \to \infty}\|u^n-x^n\|=0$ by Lemma \ref{lem:summability}, we get
\begin{align}
    \lim_{n\to\infty} \text{dist}(\textbf{0},\partial \mathcal{E}(x^n, \xi^n))=0.
\end{align}

Let $\Delta_n = \mathcal{E}(x^n, \xi^n) - \zeta = E(x^n) - \zeta$. Since $\{(x^n, \xi^n)\}_n$ is bounded and $\mathcal{E}(x^n, \xi^n) \to \zeta$, by the uniformized KL property of $\mathcal{E}(x, \xi)$, there exists a continuous concave function $\psi$ such that for all sufficiently large $n \ge N_0$,
\begin{align*}
    \psi'(\Delta_n) \cdot \dist(\textbf{0}, \partial \mathcal{E}(x^n, \xi^n)) \ge 1.
\end{align*}

By the concavity of $\psi$, we have $\psi(\Delta_n) - \psi(\Delta_{n+1}) \ge \psi'(\Delta_n)(\Delta_n - \Delta_{n+1})$. Multiplying by $C\|u^n - x^n\|$ and utilizing $\dist(\textbf{0}, \partial \mathcal{E}(x^n, \xi^n)) \le C\|u^n - x^n\|$ together with the sufficient descent $\Delta_n - \Delta_{n+1} = E(x^n) - E(x^{n+1}) \ge M\|u^n - x^n\|^2$, we obtain
\begin{align*}
    &[\psi(\Delta_n) - \psi(\Delta_{n+1})] \cdot C\|u^n - x^n\| 
    \ge [\psi(\Delta_n) - \psi(\Delta_{n+1})] \cdot \dist(\textbf{0}, \partial \mathcal{E}(x^n, \xi^n)) \\
    &\ge \underbrace{\psi'(\Delta_n) \cdot \dist(\textbf{0}, \partial \mathcal{E}(x^n, \xi^n))}_{\geq 1} \cdot (\Delta_n - \Delta_{n+1}) 
    \ge \Delta_n - \Delta_{n+1} \ge M\|u^n - x^n\|^2.
\end{align*}

Thus we arrive at
\begin{align}
    \|u^n-x^n\| \le \frac{C}{M} [\psi(\Delta_n)-\psi(\Delta_{n+1})].
\end{align}
Summing the inequality from $N_0$ to $\infty$, we have
\begin{align}
    \sum_{n=N_0}^{\infty}\|u^n-x^n\| \le \frac{C}{M} \psi(\Delta_{N_0}) < +\infty.
\end{align}
Finally, since $x^{n+1} - x^n = \lambda_n(u^n - x^n)$ and $\lambda_n \le \lambda_0$, we get
\begin{equation*}
    \sum_{n=0}^{\infty}\|x^{n+1}-x^n\|\le \sum_{n=0}^{N_0-1}\lambda_0\|u^n-x^n\|+\sum_{n=N_0}^{\infty}\lambda_0\|u^n-x^n\| <+\infty.  
\end{equation*}
\end{proof}

\subsection{Global convergence under Assumption \ref{assum:smooth_g2}}\label{sec:pb_nonsmooth_g1}
In many practical applications, the convex component $g_1$ is nonsmooth, while the concave component $-g_2$ is continuously differentiable. In this section, we establish the global convergence of Algorithm 1 under Assumption \ref{assum:smooth_g2}.
When $g_1$ is nonsmooth, the proximal optimality condition
provides a subgradient estimate at $u^n$.
We relate this estimate to the decrease at $x^{n+1}$.

\begin{lemma}\label{lem:subgrad_bound_smooth_g2}
    Under Assumption \ref{assum:smooth_g2}, let $\{x^n\}_n$ and $\{u^n\}_n$ be the sequences generated by Algorithm \ref{alg:prox_bb_dca}. Then there exists a constant $C > 0$ such that
    \begin{equation}
     \dist(\emph{\textbf{0}}, \partial E(u^n)) \le C \|u^n - x^n\|.
    \end{equation}
\end{lemma}
\begin{proof}
    From the optimality condition of the proximal subproblem \eqref{eq:algorithmic:prox}, we have
    \begin{equation*}
        \eta^n := -\nabla f(x^n) + \nabla g_2(x^n) - \frac{1}{\alpha_n}(u^n - x^n) \in \partial g_1(u^n).
    \end{equation*}
    Since $\partial E(u^n) = \nabla f(u^n) + \partial g_1(u^n) - \nabla g_2(u^n)$,
    we construct a specific subgradient $w^n \in \partial E(u^n)$ by incorporating $\eta^n \in \partial g_1(u^n)$:
    \begin{align*}
        w^n &= \nabla f(u^n) + \eta^n - \nabla g_2(u^n) \\
            &= \big( \nabla f(u^n) - \nabla f(x^n) \big) + \big( \nabla g_2(x^n) - \nabla g_2(u^n) \big) - \frac{1}{\alpha_n}(u^n - x^n).
    \end{align*}
    Noting that $x^n, u^n \in \Omega$ as established in Proposition \ref{prop:boundedness}, taking the norm and applying the triangle inequality, using the segmentwise Lipschitz moduli on $\Gamma_n$, and the step size bound $\alpha_n \ge \alpha_{\min}$, we obtain
    \begin{align*}
        \|w^n\| &\le \|\nabla f(u^n) - \nabla f(x^n)\| + \|\nabla g_2(x^n) - \nabla g_2(u^n)\| + \frac{1}{\alpha_n}\|u^n - x^n\| \\
                &\le \left( \ell_{f,n} + \ell_{g_2,n} + \frac{1}{\alpha_{\min}} \right) \|u^n - x^n\|.
    \end{align*}
    Setting the constant $C = L_{\Omega} + L_{g_2, \Omega} + \frac{1}{\alpha_{\min}} > 0$ yields 
    \begin{equation*}
     \dist(\textbf{0}, \partial E(u^n)) \le \|w^n\| \le C \|u^n - x^n\|. 
    \end{equation*}
\end{proof}

To bridge the subgradient bound at $u^n$ with the descent condition at $x^{n+1}$, we establish the following energy relation by exploiting the convexity of $f+g_1$ and the segmentwise smoothness of $g_2$ on $\Gamma_n$.

\begin{lemma}\label{lem:energy_relation_smooth_g2}
    Under Assumption \ref{assum:smooth_g2}, let $\{x^n\}_n$ and $\{u^n\}_n$ be the sequences generated by Algorithm \ref{alg:prox_bb_dca}. For any $n \ge 0$, the following energy relation holds:
    \begin{equation}\label{eq:energy_convex_combo}
        E(x^{n+1}) \le (1-\lambda_n)E(x^n) + \lambda_n E(u^n) + \frac{\ell_{g_2,n}}{2} \lambda_n(1-\lambda_n) \|u^n - x^n\|^2.
    \end{equation}
\end{lemma}
\begin{proof}
    Let $F(x) = f(x) + g_1(x)$. Since both $f$ and $g_1$ are convex, $F(x)$ is convex. Recall the line search update $x^{n+1} = x^n + \lambda_n(u^n - x^n) = (1-\lambda_n)x^n + \lambda_n u^n$. The convexity of $F$ implies
    \begin{equation}\label{eq:F_convex_combo}
        F(x^{n+1}) \le (1-\lambda_n)F(x^n) + \lambda_n F(u^n).
    \end{equation}
     Utilizing the local Lipschitz continuity of $\nabla g_2$, we have
    \begin{align}
        g_2(x^n) &\le g_2(x^{n+1}) + \langle \nabla g_2(x^{n+1}), x^n - x^{n+1} \rangle + \frac{\ell_{g_2,n}}{2}\|x^n - x^{n+1}\|^2, \label{eq:g2_xn} \\
        g_2(u^n) &\le g_2(x^{n+1}) + \langle \nabla g_2(x^{n+1}), u^n - x^{n+1} \rangle + \frac{\ell_{g_2,n}}{2}\|u^n - x^{n+1}\|^2. \label{eq:g2_un}
    \end{align}
    Multiplying \eqref{eq:g2_xn} by $(1-\lambda_n)$ and \eqref{eq:g2_un} by $\lambda_n$, and summing them up, the linear terms cancel out since $(1-\lambda_n)(x^n - x^{n+1}) + \lambda_n(u^n - x^{n+1}) = (1-\lambda_n)x^n + \lambda_n u^n - x^{n+1} = 0$. This yields:
    \begin{equation*}
        (1-\lambda_n)g_2(x^n) + \lambda_n g_2(u^n) \le g_2(x^{n+1}) + \frac{\ell_{g_2,n}}{2} \left[ (1-\lambda_n)\|x^n - x^{n+1}\|^2 + \lambda_n\|u^n - x^{n+1}\|^2 \right].
    \end{equation*}
    With $x^n - x^{n+1} = -\lambda_n(u^n - x^n)$ and $u^n - x^{n+1} = (1-\lambda_n)(u^n - x^n)$, we obtain
    \begin{equation}\label{eq:g2_concave_combo}
        -g_2(x^{n+1}) \le (1-\lambda_n)(-g_2(x^n)) + \lambda_n(-g_2(u^n)) + \frac{\ell_{g_2,n}}{2} \lambda_n(1-\lambda_n) \|u^n - x^n\|^2.
    \end{equation}
    Adding \eqref{eq:F_convex_combo} and \eqref{eq:g2_concave_combo} directly yields \eqref{eq:energy_convex_combo}, completing the proof.
\end{proof}

Equipped with the subgradient bound and the energy relation, we are now ready to establish the global convergence of the sequence via the KL property.

\begin{theorem}\label{thm:global_smooth_g2}
    Under Assumption~\ref{assum:smooth_g2}, let $\{x^n\}_n$ be a bounded sequence generated by Algorithm~\ref{alg:prox_bb_dca}. Suppose $E(x)$ is a KL function with exponent $\theta \in [0, 1)$. Then, the sequence $\{x^n\}_n$ globally converges to a critical point $x^*$ with a finite length, i.e.,
    \begin{equation*}
        \sum_{n=0}^{\infty} \|x^{n+1} - x^n\| < \infty.
    \end{equation*}
\end{theorem}
\begin{proof}
    From Lemma \ref{lem:summability}, we know that for $n \ge N_0$, $E(x^n)$ monotonically decreases to a limit $\zeta$, $\lim_{n \to \infty} \|u^n - x^n\| = 0$, and there exists a constant $M = \frac{\eta\lambda_{\min}}{2} > 0$ such that
    \begin{equation}\label{eq:new_sufficient_descent}
        E(x^n) - E(x^{n+1}) \ge M \|u^n - x^n\|^2.
    \end{equation}
    By \eqref{eq:descent_lemma} evaluated at $u^n$ and
Lemma~\ref{lem:energy_relation_smooth_g2}, using
$\lambda_n\ge\lambda_{\min}>0$ and $d^n\to0$,
we obtain $E(u^n)\to\zeta$.
The proximal optimality condition gives
\[
v^n:=\nabla g_2(x^n)-\nabla f(x^n)
-\alpha_n^{-1}d^n\in\partial g_1(u^n).
\]
For any subsequence $u^{n_j}\to\bar x$, we have
$v^{n_j}\to\nabla g_2(\bar x)-\nabla f(\bar x)
\in\partial g_1(\bar x)$ by closedness of $\partial g_1$.
Thus $\bar x$ is critical. Moreover,
\[
g_1(u^{n_j})\le g_1(\bar x)
+\langle v^{n_j},u^{n_j}-\bar x\rangle.
\]
Together with lower semicontinuity, this yields
$g_1(u^{n_j})\to g_1(\bar x)$ and $E(\bar x)=\zeta$.
Hence $E\equiv\zeta$ on the compact cluster set
$\mathcal K$ of $\{u^n\}$, and
$\dist(u^n,\mathcal K)\to0$.
The uniformized KL property therefore applies with
$\psi(s)=cs^{1-\theta}$ for some $c>0$.
    
    For any sufficiently large $n$ where $E(u^n) > \zeta$, applying the KL inequality at $u^n$ yields
    \begin{equation*}
        \psi'(E(u^n) - \zeta) \cdot \text{dist}(\textbf{0}, \partial E(u^n)) \ge 1.
    \end{equation*}
    By Lemma \ref{lem:subgrad_bound_smooth_g2}, we have $\text{dist}(\textbf{0}, \partial E(u^n)) \le C\|u^n - x^n\|$. Substituting this upper bound into the KL inequality, we obtain
    \begin{equation*}
        \psi'(E(u^n) - \zeta) \cdot C \|u^n - x^n\| \ge \psi'(E(u^n) - \zeta) \cdot \text{dist}(\textbf{0}, \partial E(u^n)) \ge 1.
    \end{equation*}
    Using the explicit form of $\psi'(E(u^n) - \zeta) = c(1-\theta)(E(u^n) - \zeta)^{-\theta}$, this becomes
    \begin{equation*}
        c(1-\theta)(E(u^n) - \zeta)^{-\theta} \cdot C \|u^n - x^n\| \ge 1.
    \end{equation*}
    Squaring both sides and rearranging the terms, we arrive at
    \begin{equation}\label{eq:new_KL_bound}
        \|u^n - x^n\|^2 \ge \frac{1}{C^2 c^2 (1-\theta)^2} (E(u^n) - \zeta)^{2\theta} \triangleq C_0 (E(u^n) - \zeta)^{2\theta}.
    \end{equation}
    
    Let $\Delta_n = E(x^n) - \zeta \ge 0$. By Lemma \ref{lem:energy_relation_smooth_g2} and noting that $1-\lambda_n < 1$, we deduce
    \begin{equation*}
        \Delta_{n+1} \le (1-\lambda_n)\Delta_n + \lambda_n(E(u^n) - \zeta) + \frac{\ell_{g_2,n}}{2}\lambda_n \|u^n - x^n\|^2.
    \end{equation*}
    Rearranging the above inequality and applying \eqref{eq:new_sufficient_descent}, we have
    \begin{equation}
        E(u^n) - \zeta
        \ge \Delta_{n+1} - \frac{1-\lambda_n}{\lambda_n}(\Delta_n - \Delta_{n+1}) - \frac{\ell_{g_2,n}}{2M}(\Delta_n - \Delta_{n+1}) 
        \ge \Delta_{n+1} - \tilde{K}(\Delta_n - \Delta_{n+1}), \label{eq:Eun_lower_bound}
    \end{equation}
    where $\tilde{K} = \frac{1-\lambda_{\min}}{\lambda_{\min}} + \frac{L_{g_2, \Omega}}{2M} > 0$. Based on the relation between $\Delta_{n+1}$ and $\Delta_n - \Delta_{n+1}$, we divide the proof into two cases.

    \textbf{Case 1:} $\tilde{K}(\Delta_n - \Delta_{n+1}) \le \frac{1}{2}\Delta_{n+1}$. 
    In this case, we have $E(u^n) - \zeta \ge \frac{1}{2}\Delta_{n+1} > 0$. Combining this with \eqref{eq:new_sufficient_descent} and \eqref{eq:new_KL_bound}, we obtain
    \begin{equation}\label{eq:case1_descent}
        \Delta_n - \Delta_{n+1} \ge M \|u^n - x^n\|^2 \ge M C_0 (E(u^n) - \zeta)^{2\theta} \ge \frac{MC_0}{2^{2\theta}} \Delta_{n+1}^{2\theta} := \gamma \Delta_{n+1}^{2\theta}.
    \end{equation}

    \textbf{Case 2:} $\tilde{K}(\Delta_n - \Delta_{n+1}) > \frac{1}{2}\Delta_{n+1}$. 
    This condition directly implies a linear decrease:
    \begin{equation}\label{eq:case2_descent}
        \Delta_{n+1} < \frac{2\tilde{K}}{2\tilde{K}+1} \Delta_n := q_2 \Delta_n, \quad \text{where } q_2 < 1.
    \end{equation}

    We proceed to analyze the finite length property based on the KL exponent $\theta \in [0, 1)$:

     (i) When $\theta = 0$: The KL property implies $\text{dist}(\textbf{0}, \partial E(u^n)) \ge \bar{c} > 0$ if $E(u^n) > \zeta$. 
    However, since $\text{dist}(\textbf{0}, \partial E(u^n)) \le C\|u^n - x^n\| \to 0$ as $n \to \infty$, we must obtain $E(u^n) \le \zeta$ for all sufficiently large $n$. 
     Thus, there exists an integer $N_1 \ge N_0$ such that $E(u^n) \le \zeta$ for all $n \ge N_1$. Consequently, the non-positive term $\lambda_n(E(u^n) - \zeta) \le 0$ in Lemma~\ref{lem:energy_relation_smooth_g2} can be omitted directly. Letting $\Delta_n = E(x^n) - \zeta \ge 0$, we obtain for all $n \ge N_1$:
\begin{equation*}
    \Delta_{n+1} \le (1-\lambda_n)\Delta_n + \frac{\ell_{g_2,n}}{2}\lambda_n(1-\lambda_n)\|u^n - x^n\|^2.
\end{equation*}
Using $\|u^n - x^n\|^2 \le \frac{1}{M}(\Delta_n - \Delta_{n+1})$ from \eqref{eq:new_sufficient_descent}, we rearrange the above inequality to obtain:
\begin{equation*}
    \Delta_{n+1} \le \left( \frac{1 - \lambda_n + \frac{\ell_{g_2,n}\lambda_n}{2M}}{1 + \frac{\ell_{g_2,n}\lambda_n}{2M}} \right) \Delta_n \le q \Delta_n, \quad \text{where } q := \frac{1 - \lambda_{\min} + \frac{L_{g_2, \Omega}\lambda_{\min}}{2M}}{1 + \frac{L_{g_2, \Omega}\lambda_{\min}}{2M}} \in (0, 1).
\end{equation*}
This implies $\Delta_n \le \Delta_{N_1} q^{n-N_1}$. Consequently, we derive
\begin{equation*}
    \sum_{n=N_1}^{\infty} \|x^{n+1} - x^n\| \le \lambda_0 \sum_{n=N_1}^{\infty} \|u^n - x^n\| \le \frac{\lambda_0}{\sqrt{M}} \sum_{n=N_1}^{\infty} \sqrt{\Delta_n}  < +\infty.
\end{equation*}
Including the finite initial steps, we conclude $\sum_{n=0}^{\infty} \|x^{n+1} - x^n\| < +\infty$.

   (ii) When $\theta \in (0, 1/2]$:
    If Case 1 occurs, \eqref{eq:case1_descent} gives $\Delta_n - \Delta_{n+1} \ge \gamma \Delta_{n+1}^{2\theta}$. Since $\Delta_n \to 0$ and $2\theta \le 1$, for sufficiently large $n$, we have $\Delta_{n+1}^{2\theta} \ge \Delta_{n+1}$. Thus, $\Delta_n - \Delta_{n+1} \ge \gamma \Delta_{n+1}$, which yields $\Delta_{n+1} \le \frac{1}{1+\gamma}\Delta_n := q_1 \Delta_n$. 
    If Case 2 occurs, \eqref{eq:case2_descent} gives $\Delta_{n+1} \le q_2 \Delta_n$. 
    Defining $q = \max\{q_1, q_2\} \in (0, 1)$, we always have $\Delta_{n+1} \le q \Delta_n$. 
    Moreover, we have $\Delta_n \le \Delta_{N_0} q^{n-N_0}$. Since $\Delta_{n+1} \ge 0$, we get $\sqrt{\Delta_n - \Delta_{n+1}} \le  \sqrt{\Delta_n}$, which further brings
    \begin{equation*}
        \sum_{n=N_0}^{\infty} \sqrt{\Delta_n - \Delta_{n+1}} \le \sum_{n=N_0}^{\infty} \sqrt{\Delta_n} \le \sqrt{\Delta_{N_0}} \sum_{n=N_0}^{\infty} (\sqrt{q})^{n-N_0} < +\infty.
    \end{equation*}

    (iii) When $\theta \in (1/2, 1)$: 
    Considering the concave function $h(s) = cs^{1-\theta}$, we have
    \begin{equation}\label{eq:concave_ineq}
        \Delta_n^{1-\theta} - \Delta_{n+1}^{1-\theta} \ge (1-\theta)\Delta_n^{-\theta}(\Delta_n - \Delta_{n+1}).
    \end{equation}
    If Case 1 occurs, the condition $\tilde{K}(\Delta_n - \Delta_{n+1}) \le \frac{1}{2}\Delta_{n+1}$ implies $\Delta_n \le (1 + \frac{1}{2\tilde{K}})\Delta_{n+1} := \mu \Delta_{n+1}$ ($\mu \ge 1$), which gives $\Delta_n^{-\theta} \ge \mu^{-\theta}\Delta_{n+1}^{-\theta}$. Moreover, \eqref{eq:case1_descent} yields $\sqrt{\Delta_n - \Delta_{n+1}} \ge \sqrt{\gamma}\Delta_{n+1}^\theta$. Dividing \eqref{eq:concave_ineq} by $\sqrt{\Delta_n - \Delta_{n+1}}$, we obtain
    \begin{align*}
        \frac{\Delta_n^{1-\theta} - \Delta_{n+1}^{1-\theta}}{\sqrt{\Delta_n - \Delta_{n+1}}} &\ge (1-\theta)\mu^{-\theta} \Delta_{n+1}^{-\theta} \sqrt{\Delta_n - \Delta_{n+1}} \\
        &\ge (1-\theta)\mu^{-\theta} \Delta_{n+1}^{-\theta} \sqrt{\gamma} \Delta_{n+1}^\theta = (1-\theta)\mu^{-\theta}\sqrt{\gamma} := \hat{C} > 0.
    \end{align*}
    Thus, $\sqrt{\Delta_n - \Delta_{n+1}} \le \frac{1}{\hat{C}} (\Delta_n^{1-\theta} - \Delta_{n+1}^{1-\theta})$. 
    If Case 2 occurs, \eqref{eq:case2_descent} gives $\Delta_{n+1} < q_2 \Delta_n$. Then we have
    \begin{equation*}
        \frac{\sqrt{\Delta_n - \Delta_{n+1}}}{\Delta_n^{1-\theta} - \Delta_{n+1}^{1-\theta}} \le \frac{\sqrt{\Delta_n}}{\Delta_n^{1-\theta} - q_2^{1-\theta}\Delta_n^{1-\theta}} = \frac{1}{1-q_2^{1-\theta}} \Delta_n^{\theta - 1/2}.
    \end{equation*}
    Since $\theta > 1/2$ and $\Delta_n \to 0$, we have $\Delta_n^{\theta - 1/2} \to 0$. Thus, for sufficiently large $n$, there exists $C_2 > 0$ such that $\sqrt{\Delta_n - \Delta_{n+1}} \le C_2 (\Delta_n^{1-\theta} - \Delta_{n+1}^{1-\theta})$.
   Combining both cases, there exists $C_3 = \max\{1/\hat{C}, C_2\} > 0$ such that for all large $n$,
    \begin{equation*}
        \sqrt{\Delta_n - \Delta_{n+1}} \le C_3 (\Delta_n^{1-\theta} - \Delta_{n+1}^{1-\theta}).
    \end{equation*}
    Summing the series from $n = N_0$ to $\infty$, we get
    \begin{equation*}
        \sum_{n=N_0}^{\infty} \sqrt{\Delta_n - \Delta_{n+1}} \le C_3 \sum_{n=N_0}^{\infty} (\Delta_n^{1-\theta} - \Delta_{n+1}^{1-\theta}) \le C_3 \Delta_{N_0}^{1-\theta} < +\infty.
    \end{equation*}

    Finally, across all $\theta \in [0, 1)$, we have established that $\sum_{n=N_0}^{\infty} \sqrt{\Delta_n - \Delta_{n+1}} < \infty$. Given the line search update $x^{n+1} - x^n = \lambda_n(u^n - x^n)$ with $\lambda_n \le \lambda_0$, and applying the sufficient descent condition \eqref{eq:new_sufficient_descent}, we conclude
    \begin{align*}
        \sum_{n=N_0}^{\infty} \|x^{n+1} - x^n\| &\le \lambda_0 \sum_{n=N_0}^{\infty} \|u^n - x^n\| 
        \le \frac{\lambda_0}{\sqrt{M}} \sum_{n=N_0}^{\infty} \sqrt{\Delta_n - \Delta_{n+1}} < +\infty.
    \end{align*}
    Including the finite initial steps, we ultimately have $\sum_{n=0}^{\infty} \|x^{n+1} - x^n\| < +\infty$. 
\end{proof}

\subsection{Local convergence}
We first derive local convergence rates under Assumption~\ref{assum:smooth_g1}. In this setting, the relevant KL exponent belongs to the conjugate auxiliary function $\mathcal{E}$ in
\eqref{eq:conj_potential}. 
Building upon the convergence rate frameworks established in \cite{Attouch2009, Artacho2018}, we establish the local convergence rates of pBBDCA as follows.

\begin{theorem}[Local convergence rate under Assumption~\ref{assum:smooth_g1}]
\label{thm:alg1_rate_smooth_g1}
Suppose the assumptions of Theorem~\ref{global} hold and
$x^n\to x^*$. Set
$\xi^*:=\nabla f(x^*)+\nabla g_1(x^*)$.
Assume that $\zeta:=\lim_{n\to\infty}E(x^n)=E(x^*)$
and that $\mathcal{E}$ has the KL property with exponent
$\theta\in[0,1)$ at $(x^*,\xi^*)$. Then:
\begin{enumerate}[label=\textup{(\roman*)}]
\item If $\theta\in[0,1/2)$, Algorithm~\ref{alg:prox_bb_dca}
terminates after finite iterations.
\item If $\theta=1/2$, there exist $A>0$ and $q\in(0,1)$
such that $\|x^n-x^*\|\le A q^n$ for all sufficiently large $n$.
\item If $\theta\in(1/2,1)$, there exists $A>0$ such that
$\|x^n-x^*\|\le A n^{-(1-\theta)/(2\theta-1)}$
for all sufficiently large $n$.
\end{enumerate}
\end{theorem}

\begin{proof}
Write $d^n=u^n-x^n$ and $\Delta_n=E(x^n)-\zeta$.
By \eqref{eq:eventual_descent}, for all sufficiently large $n$,
\begin{equation}\label{eq:alg1_rate_g1_descent}
    \Delta_n-\Delta_{n+1}\ge M\|d^n\|^2,
    \qquad M>0.
\end{equation}
Moreover, the optimality condition for
\eqref{eq:algorithmic:prox} yields
\[
    \xi^n
    =\nabla f(x^n)+\nabla g_1(u^n)+\alpha_n^{-1}d^n.
\]
Since $x^n\to x^*$, $d^n\to0$, and
$\alpha_n\ge\alpha_{\min}>0$, we have $\xi^n\to\xi^*$.
The closedness of $\partial g_2$ then gives
$\xi^*\in\partial g_2(x^*)$, and hence
$\mathcal{E}(x^*,\xi^*)=E(x^*)=\zeta$.

If $\Delta_N=0$ for some sufficiently large $N$, then
\eqref{eq:alg1_rate_g1_descent} and $\Delta_{N+1}\ge0$
imply $d^N=0$, so the algorithm terminates.
It therefore suffices to consider the case
$\Delta_n>0$ for all sufficiently large $n$.

By the KL exponent inequality at $(x^*,\xi^*)$ and
the relative subgradient bound in the proof of
Theorem~\ref{global}, there exist $\kappa,C>0$ such that
\[
    \kappa\Delta_n^\theta
    \le \dist(\mathbf{0},\partial\mathcal{E}(x^n,\xi^n))
    \le C\|d^n\|.
\]
Combining this with \eqref{eq:alg1_rate_g1_descent} gives
\begin{equation}\label{eq:alg1_rate_g1_rec}
    \Delta_n-\Delta_{n+1}\ge a\Delta_n^{2\theta},
    \qquad a:=\frac{M\kappa^2}{C^2}>0.
\end{equation}

Set $\psi(s)=s^{1-\theta}/[\kappa(1-\theta)]$.
Applying the single-step estimate from the proof of
Theorem~\ref{global} with this $\psi$, we obtain
\[
    \|d^k\|
    \le \frac{C}{M}
    \bigl[\psi(\Delta_k)-\psi(\Delta_{k+1})\bigr].
\]
Summing from $k=n$ to infinity yields
\begin{equation}\label{eq:tail_length_alg1}
    \sum_{k=n}^{\infty}\|d^k\|
    \le \frac{C}{M}\psi(\Delta_n).
\end{equation}
Since $x^{k+1}-x^k=\lambda_kd^k$ and
$\lambda_k\le\lambda_0$, it follows that
\begin{equation}\label{eq:iterate_error_alg1}
    \|x^n-x^*\|
    \le \lambda_0\sum_{k=n}^{\infty}\|d^k\|
    \le \frac{\lambda_0C}{M\kappa(1-\theta)}
        \Delta_n^{1-\theta}.
\end{equation}
Choose $N_0$ sufficiently large so that the preceding
estimates hold for all $n\ge N_0$.

(i) $0\le\theta<1/2$.
By \eqref{eq:alg1_rate_g1_rec} and $\Delta_{n+1}\ge0$,
\[
    a\Delta_n^{2\theta}
    \le \Delta_n-\Delta_{n+1}
    \le \Delta_n.
\]
Thus $\Delta_n^{1-2\theta}\ge a>0$, contradicting
$\Delta_n\to0$ because $1-2\theta>0$.
Therefore, $\Delta_N=0$ for some finite $N$,
and the algorithm terminates.

 (ii) $\theta=1/2$.
In this case, \eqref{eq:alg1_rate_g1_rec} becomes
\[
    0\le\Delta_{n+1}\le(1-a)\Delta_n.
\]
If $a\ge1$, finite termination follows.
Otherwise, setting $q=1-a\in(0,1)$ gives
\[
    \Delta_n\le\Delta_{N_0}q^{n-N_0}.
\]
Consequently, \eqref{eq:iterate_error_alg1} yields
\[
    \|x^n-x^*\|
    \le \frac{2\lambda_0C}{M\kappa}
        \Delta_{N_0}^{1/2}q^{(n-N_0)/2}
    =O\bigl((\sqrt q)^n\bigr).
\]
In this paper, we use $O$ to denote the asymptotic upper bound.

 (iii) $1/2<\theta<1$.
Set $r=2\theta-1>0$.
By the convexity of $t\mapsto t^{-r}$ and
\eqref{eq:alg1_rate_g1_rec},
\[
    \Delta_{n+1}^{-r}-\Delta_n^{-r}
    \ge r\Delta_n^{-r-1}(\Delta_n-\Delta_{n+1})
    \ge ra.
\]
Summing from $N_0$ to $n-1$ gives
\[
    \Delta_n^{-r}
    \ge \Delta_{N_0}^{-r}+ra(n-N_0),
\]
and therefore
\[
    \Delta_n
    \le
    \bigl[\Delta_{N_0}^{-r}+ra(n-N_0)\bigr]^{-1/r}
    =O(n^{-1/r}).
\]
Using \eqref{eq:iterate_error_alg1}, we conclude that
\[
    \|x^n-x^*\|
    =O\bigl(\Delta_n^{1-\theta}\bigr)
    =O\left(n^{-\frac{1-\theta}{2\theta-1}}\right).
\]
\end{proof}
When $g_1$ may be nonsmooth and $g_2$ is smooth, the KL
inequality is evaluated at the proximal trial point $u^n$. The local convergence arises from Theorem~\ref{thm:global_smooth_g2}.

\begin{theorem}[Local convergence rate under Assumption~\ref{assum:smooth_g2}]
\label{thm:alg1_rate_smooth_g2}
Suppose the assumptions of Theorem~\ref{thm:global_smooth_g2}
hold, the sequence $x^n$ converges to $x^*$, and
$\zeta:=\lim_{n\to\infty}E(x^n)=E(x^*)$.
Suppose $E$ has KL exponent $\theta\in[0,1)$ at $x^*$.
Then:
\begin{enumerate}[label=\textup{(\roman*)}]
\item If $\theta\in[0,1/2]$, there exist $A>0$ and
$q\in(0,1)$ such that $\|x^n-x^*\|\le A q^n$
for all sufficiently large $n$.
\item If $\theta\in(1/2,1)$, there exists $A>0$ such
that $\|x^n-x^*\|\le A n^{-(1-\theta)/(2\theta-1)}$
for all sufficiently large $n$.
\end{enumerate}
\end{theorem}

\begin{proof}
Define $\Delta_n=E(x^n)-\zeta$ and
$D_n=\Delta_n-\Delta_{n+1}$.
By \eqref{eq:new_sufficient_descent}, for all sufficiently
large $n$,
\begin{equation}\label{eq:alg1_rate_g2_step}
D_n\ge M\|u^n-x^n\|^2,
\qquad
\|x^{n+1}-x^n\|
\le\frac{\lambda_0}{\sqrt M}\sqrt{D_n}.
\end{equation}
If $\Delta_N=0$ for a sufficiently large $N$,
\eqref{eq:alg1_rate_g2_step} implies that the method
terminates. We may therefore assume $\Delta_n>0$
for all sufficiently large $n$.
In the proof of
Theorem~\ref{thm:global_smooth_g2}, we have considered two cases:
\begin{equation}\label{eq:alg1_rate_g2_cases}
\begin{cases}
\text{Case 1:}&D_n\ge\gamma\Delta_{n+1}^{2\theta},\\
\text{Case 2:}&\Delta_{n+1}\le q_2\Delta_n,
\end{cases}
\qquad\gamma>0,\quad q_2\in(0,1).
\end{equation}

 (i) $0\le\theta\le1/2$.
For all sufficiently large $n$ we have
$0<\Delta_{n+1}\le1$. If Case 1 holds, then
\[
D_n\ge\gamma\Delta_{n+1}^{2\theta}
\ge\gamma\Delta_{n+1},
\quad\text{so}\quad
\Delta_{n+1}\le(1+\gamma)^{-1}\Delta_n.
\]
In Case 2, the second alternative in
\eqref{eq:alg1_rate_g2_cases} applies.
Hence, with
$q:=\max\{(1+\gamma)^{-1},q_2\}<1$,
we have $\Delta_{n+1}\le q\Delta_n$ for all large $n$.
Consequently $\Delta_n=O(q^n)$ and, by
\eqref{eq:alg1_rate_g2_step} and the finite-length property,
\[
\|x^n-x^*\|
\le\frac{\lambda_0}{\sqrt M}
\sum_{k=n}^{\infty}\sqrt{D_k}
\le\frac{\lambda_0}{\sqrt M}
\sum_{k=n}^{\infty}\sqrt{\Delta_k}
=O(q^{n/2}).
\]

(ii) $1/2<\theta<1$.
Put $p=2\theta>1$, $r=p-1>0$, and
$\mu=1+(2\widetilde K)^{-1}>1$.
In Case 1, the defining inequality gives
$\Delta_n\le\mu\Delta_{n+1}$.
Consequently,
\begin{align*}
\Delta_{n+1}^{-r}-\Delta_n^{-r}
\ge r\Delta_n^{-p}D_n\ge r\gamma\mu^{-p}>0.
\end{align*}
In Case 2, the second inequality in
\eqref{eq:alg1_rate_g2_cases} yields
\[
\Delta_{n+1}^{-r}-\Delta_n^{-r}
\ge(q_2^{-r}-1)\Delta_n^{-r}
\ge q_2^{-r}-1>0
\]
for all sufficiently large $n$ when $\Delta_n\le1$.
Therefore, there is a constant $b>0$ such that
$\Delta_{n+1}^{-r}-\Delta_n^{-r}\ge b$ in both cases. Summing from $N_0$ to $n-1$
yields
\[
    \Delta_n^{-r}\ge\Delta_{N_0}^{-r}+b(n-N_0).
\]
Consequently,
\[
    \Delta_n
    \le\bigl[\Delta_{N_0}^{-r}+b(n-N_0)\bigr]^{-1/r}
    =O(n^{-1/r}).
\]
Now we show that, for some $B>0$ and all sufficiently large $n$,
\[
    \sqrt{D_n}
    \le B\bigl(\Delta_n^{1-\theta}
               -\Delta_{n+1}^{1-\theta}\bigr).
\]
In Case 1, concavity, $\Delta_n\le\mu\Delta_{n+1}$, and
$D_n\ge\gamma\Delta_{n+1}^{2\theta}$ give
\begin{align*}
    \Delta_n^{1-\theta}-\Delta_{n+1}^{1-\theta}
    &\ge (1-\theta)\Delta_n^{-\theta}D_n
    \ge (1-\theta)\mu^{-\theta}
          \Delta_{n+1}^{-\theta}D_n\\
    &\ge (1-\theta)\mu^{-\theta}\sqrt{\gamma}\sqrt{D_n}.
\end{align*}
In Case 2, $\Delta_{n+1}\le q_2\Delta_n$ and $D_n\le\Delta_n$
imply
\[
    \frac{\sqrt{D_n}}
         {\Delta_n^{1-\theta}-\Delta_{n+1}^{1-\theta}}
    \le
    \frac{\sqrt{\Delta_n}}
         {(1-q_2^{1-\theta})\Delta_n^{1-\theta}}
    =
    \frac{\Delta_n^{\theta-1/2}}{1-q_2^{1-\theta}}.
\]
This is uniformly bounded since $\theta>1/2$ and $\Delta_n\to0$,
so the claimed inequality holds in both cases.
Summing from $k=n$ to $N$ and letting $N\to\infty$ yields
\[
    \sum_{k=n}^{\infty}\sqrt{D_k}
    \le B\Delta_n^{1-\theta}.
\]
Consequently, by the descent estimate and the update rule,
\begin{align*}
    \|x^n-x^*\|
    &\le\sum_{k=n}^{\infty}\|x^{k+1}-x^k\|\le\frac{\lambda_0}{\sqrt M}
          \sum_{k=n}^{\infty}\sqrt{D_k}\le\frac{\lambda_0B}{\sqrt M}\Delta_n^{1-\theta}
     =O\left(n^{-\frac{1-\theta}{2\theta-1}}\right).
\end{align*}
\end{proof}

\section{Convergence Analysis of pABBDCA$_{\text{er}}$}\label{sec:alg2_conv}

To address the non-monotonicity introduced by extrapolation, this section establishes the convergence properties of pABBDCA$_{\text{er}}$ (Algorithm \ref{alg:prox_bb_dca_ext_res}), which relies on a heuristic restart mechanism to recover theoretical stability. We first develop a sequence of foundational lemmas for the subsequent Kurdyka-\L ojasiewicz (KL) convergence analysis.

\begin{lemma}\label{lem:step_lower_bound}

 Let $\{x^n\}_n$ be the sequence generated by Algorithm \ref{alg:prox_bb_dca_ext_res}. Then, the accepted step size $\lambda_n$ is uniformly bounded away from zero. That is, there exists a constant $\lambda_{\min} > 0$ such that $\lambda_n \ge \lambda_{\min}$ for all $n \ge 0$.

\end{lemma}

\begin{proof}
For any iteration $n$, the nonmonotone line search terminates under two exclusive cases:

 (i) \textbf{Successful Extrapolation:} The line search condition is satisfied within the maximum allowed iterations $N_{\max}$. In this case, we have
$$\lambda_n \ge \lambda_0 \rho^{N_{\max}}.$$

(ii) \textbf{Restart Triggered:} In this case, the line search fails after $N_{\max}$ iterations and the restart mechanism is triggered. Now that $\beta_n = 0$, we obtain $y^n = x^n$. 
According to Lemma \ref{lem:well-definedness}, the nonmonotone line search will terminate in a finite number of steps and there exists a constant $\tilde{\lambda}_{\min} > 0$ such that $\lambda_n \ge \tilde{\lambda}_{\min}$ for all $n \ge 0$.

Combining both cases, we define $$\lambda_{\min} = \min\{\lambda_0 \rho^{N_{\max}},\tilde{\lambda}_{\min}\} > 0.$$

It follows that $\lambda_n \ge \lambda_{\min}$ holds globally for all $n \ge 0$.
\end{proof}

\begin{lemma}\label{lem:sufficient_descent}

Let $\{x^n\}_n$ be the sequence generated by Algorithm \ref{alg:prox_bb_dca_ext_res}. There exists an integer $N_0 \ge 0$ such that for all $n \ge N_0$, the sequence of energy $\{E(x^n)\}$ is monotonically decreasing. Furthermore, we have $$\sum_{n=0}^{\infty} \|x^{n+1} - x^n\|^2 < \infty.$$
\end{lemma}

\begin{proof}
The proof follows identical arguments to those in Lemma \ref{lem:summability}, utilizing the nonmonotone line search condition and the uniform step size lower bound established in Lemma \ref{lem:step_lower_bound}. To avoid redundancy, the detailed algebraic derivations are omitted here.
\end{proof}

As a direct consequence of the eventual monotonic descent established in Lemma \ref{lem:sufficient_descent}, we can deduce the boundedness of the algorithmic trajectory, which follows the same argument as the geometric confinement analyzed in Proposition \ref{prop:boundedness}.

\begin{corollary}\label{cor:boundedness_alg2}
The sequences $\{x^n\}_n$, $\{y^n\}_n$, and $\{u^n\}_n$ generated by Algorithm \ref{alg:prox_bb_dca_ext_res} are bounded. Consequently, there exists a convex compact set $\Omega$ encompassing all these iterates, over which $\nabla f$ is Lipschitz continuous on $\Omega$ with constant $L_{\Omega}$. Furthermore, depending on the regularity conditions, $\nabla g_1$ is uniformly locally Lipschitz continuous with constant $L_{g_1, \Omega}$ (under Assumption \ref{assum:smooth_g1}), or $\nabla g_2$ is uniformly locally Lipschitz continuous with constant $L_{g_2, \Omega}$ (under Assumption \ref{assum:smooth_g2}).
\end{corollary}

\begin{proof}
The boundedness of $\{x^n\}_n$ follows directly from the eventual monotonic decrease of $E(x^n)$ and the level-boundedness assumption. The extrapolated sequence $y^n = x^n + \beta_n(x^n - x^{n-1})$ is bounded since $\beta_n < 1$. The boundedness of $\{u^n\}_n$ and the extraction of local Lipschitz constants via the finite subcover theorem follow the identical arguments as in Proposition \ref{prop:boundedness}. 
\end{proof}

\subsection{Global convergence under Assumption \ref{assum:smooth_g1}}
We are now in a position to establish the global convergence theorem.

\begin{theorem}\label{thm:global_alg2}
Let $\{x^{n}\}_n$ be generated by Algorithm \ref{alg:prox_bb_dca_ext_res} for solving \eqref{eq:basefunctional}. Under Assumption \ref{assum:smooth_g1}, if $\mathcal{E}(x, \xi)$ is a KL function, then the sequence $\{x^{n}\}_n$ globally converges to a critical point $x^*$ with a finite trajectory length:
\begin{equation}
    \sum_{n=0}^{\infty}\|x^{n+1}-x^{n}\| < \infty.
\end{equation}
\end{theorem}

\begin{proof}
Following identical arguments to those in Theorem \ref{global}, the sequences $\{x^n\}_n$ and $\{\xi^n\}_n$ are bounded, and any cluster point $x^*$ is a critical point of problem \eqref{eq:basefunctional}. 

To establish the global convergence via the KL property of $\mathcal{E}(x, \xi)$, we evaluate its limiting subdifferential at $(x^n, \xi^n)$. Recall that $\xi^n \in \partial g_2(x^n)$ implies $x^n \in \partial g_2^*(\xi^n)$ and $\mathcal{E}(x^n, \xi^n) = E(x^n)$. From the optimality condition of the extrapolated subproblem \eqref{eq:algorithm2:prox}, we have
\begin{equation*}
    \nabla f(u^n) + \nabla g_1(u^n) - \xi^n + \frac{1}{\alpha_n}(u^n - y^n) = \textbf{0}.
\end{equation*}
With $\textbf{0} \in -x^n + \partial g_2^*(\xi^n)$, we obtain
\begin{align*}
    \begin{pmatrix}
        (\nabla f(x^n) - \nabla f(u^n)) + (\nabla g_1(x^n) - \nabla g_1(u^n)) - \frac{1}{\alpha_n}(u^n - y^n) \\
        \textbf{0}
    \end{pmatrix} \in \partial \mathcal{E}(x^n, \xi^n).
\end{align*}
Noting that $u^n - y^n = u^n - x^n - \beta_n(x^n - x^{n-1})$ and $x^n - x^{n-1} = \lambda_{n-1}(u^{n-1} - x^{n-1})$ with $\lambda_{n-1} \le 1$, we utilize the local Lipschitz constants on the compact set $\Omega$ to bound the subgradient as
\begin{align*}
  \dist(\textbf{0},\partial \mathcal{E}(x^n, \xi^n)) &\le \|\nabla f(x^n) - \nabla f(u^n)\| + \|\nabla g_1(x^n) - \nabla g_1(u^n)\| + \left\|\frac{1}{\alpha_n}(u^n-y^n) \right\| \\
    &\le \left(\ell_{f,n} + \ell_{g_1,n} + \frac{1}{\alpha_{\min}}\right)\|u^n-x^n\| + \frac{\beta_n}{\alpha_{\min}}\|x^n-x^{n-1}\| \\
    &\le C(\|u^n-x^n\| + \|u^{n-1}-x^{n-1}\|),
\end{align*}
where $C = L_{\Omega} + L_{g_1, \Omega} + \frac{1}{\alpha_{\min}}> 0$. Since $\|u^n-x^n\|\to 0$, we have $\text{dist}(\textbf{0},\partial \mathcal{E}(x^n, \xi^n)) \to 0$.

Let $\Delta_n = \mathcal{E}(x^n, \xi^n) - \zeta = E(x^n) - \zeta > 0$. Let $\psi$ be the continuous concave function given by the uniformized KL property of $\mathcal{E}$, and define $\phi_n = \psi(\Delta_n) - \psi(\Delta_{n+1})$. The concavity of $\psi$ implies that $\phi_n \ge \psi'(\Delta_n)(\Delta_n - \Delta_{n+1})$. Applying the KL inequality $\psi'(\Delta_n) \cdot \text{dist}(\textbf{0}, \partial \mathcal{E}(x^n, \xi^n)) \ge 1$, we obtain
\begin{align*}
    \phi_n \cdot \text{dist}(\textbf{0},\partial \mathcal{E}(x^n, \xi^n)) 
    &\ge \psi'(\Delta_n) \cdot \text{dist}(\textbf{0},\partial \mathcal{E}(x^n, \xi^n)) \cdot [E(x^n) - E(x^{n+1})] \\
    &\ge E(x^n) - E(x^{n+1}) \ge M\|u^n-x^n\|^2,
\end{align*}
where $M = \frac{\eta \lambda_{\min}}{2}$. Combining the above bounds and denoting $K = \frac{C}{M}$, we arrive at
\begin{equation}
    \|u^n-x^n\|^2 \le K \phi_n (\|u^n-x^n\| + \|u^{n-1}-x^{n-1}\|).
\end{equation}
Applying the inequality $a^2 \le cd \implies a \le c + \frac{d}{4}$, we obtain
\begin{equation}
    \frac{1}{2}\|u^n-x^n\| \le K\phi_n + \frac{1}{4}(\|u^{n-1}-x^{n-1}\| - \|u^n-x^n\|).
\end{equation}
Summing this inequality from $N_0$ to $\infty$ yields
\begin{equation}
    \frac{1}{2}\sum_{n=N_0}^{\infty}\|u^n-x^n\| \le K \psi(\Delta_{N_0}) + \frac{1}{4}\|u^{N_0-1}-x^{N_0-1}\| < +\infty.
\end{equation}
Given $\|x^{n+1}-x^n\| = \lambda_n\|u^n-x^n\| \le \lambda_0\|u^n-x^n\|$, we conclude $\sum_{n=0}^{\infty}\|x^{n+1}-x^n\| < +\infty$.
\end{proof}

\subsection{Global convergence under Assumption \ref{assum:smooth_g2}}

In this section, we establish the global convergence of Algorithm \ref{alg:prox_bb_dca_ext_res} (pABBDCA$_{\text{er}}$) under Assumption \ref{assum:smooth_g2}. As established in Corollary \ref{cor:boundedness_alg2}, $\nabla g_2$ is uniformly $L_{g_2, \Omega}$-Lipschitz continuous over the compact set $\Omega$. We first establish a subgradient bound that controls the subgradient of the energy function by the step variations.

\begin{lemma}\label{lem:subgrad_bound_alg2_smooth_g2}
    Under Assumption \ref{assum:smooth_g2}, let $\{x^n\}_n$ and $\{u^n\}_n$ be the sequences generated by Algorithm \ref{alg:prox_bb_dca_ext_res}. Then, there exists a constant $C > 0$ such that
    \begin{equation}
      \dist(\emph{\textbf{0}}, \partial E(u^n)) \le C \left( \|u^n - x^n\| + \|x^n - x^{n-1}\| \right).
    \end{equation}
\end{lemma}

\begin{proof}
    From the optimality condition of the subproblem \eqref{eq:algorithm2:prox}, and noting that $\xi^n = \nabla g_2(x^n)$, we have
    \begin{equation}
        \eta^n := \nabla g_2(x^n) - \nabla f(u^n) - \frac{1}{\alpha_n}(u^n - y^n) \in \partial g_1(u^n).
    \end{equation}
    The subdifferential of the energy function evaluated at the subproblem minimizer $u^n$ is given by $\partial E(u^n) = \nabla f(u^n) + \partial g_1(u^n) - \nabla g_2(u^n)$. By selecting  $\eta^n \in \partial g_1(u^n)$, we construct $w^n \in \partial E(u^n)$ as
    \begin{equation*}
        w^n = \nabla f(u^n) + \eta^n - \nabla g_2(u^n) 
            = \nabla g_2(x^n) - \nabla g_2(u^n) - \frac{1}{\alpha_n}(u^n - y^n),
    \end{equation*}
    where the implicit gradient $\nabla f(u^n)$ cancels out. Utilizing the extrapolation formula $y^n = x^n + \beta_n(x^n - x^{n-1})$, we bound the deviation as
    \begin{equation}
        \|u^n - y^n\| \le \|u^n - x^n\| + \beta_n\|x^n - x^{n-1}\|.
    \end{equation}
    Taking the norm of $w^n$, together with the step size bound $\alpha_n \ge \alpha_{\min}$, and the extrapolation upper bound $\beta_n <1$, we obtain
    \begin{align*}
        \|w^n\| &\le \|\nabla g_2(x^n) - \nabla g_2(u^n)\| + \frac{1}{\alpha_n}\|u^n - y^n\| \\ &\le \ell_{g_2,n}\|u^n - x^n\| + \frac{1}{\alpha_{\min}}\left( \|u^n - x^n\| + \|x^n - x^{n-1}\| \right)  \\
                &= \left( \ell_{g_2,n} + \frac{1}{\alpha_{\min}} \right) \|u^n - x^n\| + \frac{1}{\alpha_{\min}} \|x^n - x^{n-1}\|.
    \end{align*}
    Setting the constant $C=L_{g_2,\Omega}+\frac{1}{\alpha_{\min}}>0$, we conclude
    \begin{equation}
      \dist(\textbf{0}, \partial E(u^n)) \le C \left( \|u^n - x^n\| + \|x^n - x^{n-1}\| \right),
    \end{equation}
    which completes the proof.
\end{proof}

Equipped with Lemma \ref{lem:sufficient_descent} and Lemma \ref{lem:subgrad_bound_alg2_smooth_g2}, the theoretical framework for addressing the nonsmooth $g_1$ component is now complete. By synthesizing the energy-gap analysis involving $E(u^n)$ established in Theorem \ref{thm:global_smooth_g2} with the KL algebraic decoupling techniques demonstrated in Theorem \ref{thm:global_alg2}, the global convergence of pABBDCA$_{\text{er}}$ follows naturally. 

\begin{theorem}\label{thm:global_alg2_smooth_g2}
    Under Assumption \ref{assum:smooth_g2}, let $\{x^n\}_n$ be the bounded sequence generated by Algorithm \ref{alg:prox_bb_dca_ext_res}. Suppose $E(x)$ is a KL function with exponent $\theta \in [0, 1)$. Then, the sequence $\{x^n\}_n$ globally converges to a critical point $x^*$ with a finite trajectory length, i.e.,
    \begin{equation}
        \sum_{n=0}^{\infty} \|x^{n+1} - x^n\| < \infty.
    \end{equation}
\end{theorem}

\begin{proof}
    From Lemma \ref{lem:sufficient_descent}, we know that for $n \ge N_0$, $E(x^n)$ monotonically decreases to a limit $\zeta = E(x^*)$, and $\lim_{n \to \infty} \|u^n - x^n\| = 0$. Moreover, there exists a constant $M = \frac{\eta\lambda_{\min}}{2} > 0$ such that
    \begin{equation}\label{eq:alg2_suff_descent}
        E(x^n) - E(x^{n+1}) \ge M \|u^n - x^n\|^2.
    \end{equation}
    The proximal optimality condition gives
\[
E(u^n)-E(x^n)
\le-\frac{\|d^n\|^2}{\alpha_n}
+\frac{\beta_n}{\alpha_n}
 \langle d^n,x^n-x^{n-1}\rangle.
\]
Together with Lemma~\ref{lem:energy_relation_smooth_g2}
and $\lambda_n\ge\lambda_{\min}>0$, this yields
$E(u^n)\to\zeta$.
Using
\[
v^n:=\nabla g_2(x^n)-\nabla f(u^n)
-\alpha_n^{-1}(u^n-y^n)\in\partial g_1(u^n),
\]
the cluster-set argument in the proof of
Theorem~\ref{thm:global_smooth_g2} shows that every
cluster point of $\{u^n\}$ is critical and has value
$\zeta$.
Let $\Delta_n=E(x^n)-\zeta$.
The uniformized KL property gives $c>0$ such that,
for all sufficiently large $n$ with $E(u^n)>\zeta$,
    \begin{equation*}
        c(1-\theta)(E(u^n) - \zeta)^{-\theta} \cdot \text{dist}(\textbf{0}, \partial E(u^n)) \ge 1.
    \end{equation*}
    
    Substituting the subgradient bound from Lemma \ref{lem:subgrad_bound_alg2_smooth_g2} into the KL inequality and squaring both sides, we use $(a+b)^2 \le 2(a^2+b^2)$ to obtain
    \begin{equation}\label{eq:alg2_kl_squared}
        (E(u^n) - \zeta)^{2\theta} \le C^2 c^2 (1-\theta)^2 \cdot 2\left( \|u^n - x^n\|^2 + \|x^n - x^{n-1}\|^2 \right).
    \end{equation}
    Notice that the line search update gives $x^n - x^{n-1} = \lambda_{n-1}(u^{n-1} - x^{n-1})$. Since $\lambda_{n-1} \le \lambda_0$, we can express the squared distances entirely in terms of the energy gaps using \eqref{eq:alg2_suff_descent}:
    \begin{align*}
        \|u^n - x^n\|^2 &\le \frac{1}{M}(\Delta_n - \Delta_{n+1}), \\
        \|x^n - x^{n-1}\|^2 &\le \lambda_0^2 \|u^{n-1} - x^{n-1}\|^2 \le \frac{1}{M}(\Delta_{n-1} - \Delta_n).
    \end{align*}
    Substituting these relations into \eqref{eq:alg2_kl_squared} and defining $C_1 = \frac{2 C^2 c^2 (1-\theta)^2}{M}$, we arrive at
    \begin{equation}\label{eq:alg2_kl_gap}
        (E(u^n) - \zeta)^{2\theta} \le C_1 \left( (\Delta_n - \Delta_{n+1}) + (\Delta_{n-1} - \Delta_n) \right) = C_1 (\Delta_{n-1} - \Delta_{n+1}).
    \end{equation}

    Since the line search update $x^{n+1} = (1-\lambda_n)x^n + \lambda_n u^n$ remains identical to Algorithm \ref{alg:prox_bb_dca}, the convex combination energy relation established in Lemma \ref{lem:energy_relation_smooth_g2} also holds. Rearranging the terms, we have
    \begin{equation*}
        E(u^n) - \zeta \ge \Delta_{n+1} - \tilde{K}(\Delta_n - \Delta_{n+1}),
    \end{equation*}
    where $\tilde{K} = \frac{1-\lambda_{\min}}{\lambda_{\min}} + \frac{L_{g_2,\Omega}}{2M} > 0$. Since $\Delta_n \le \Delta_{n-1}$ for all $n \ge N_0$, it follows that $\Delta_n - \Delta_{n+1} \le \Delta_{n-1} - \Delta_{n+1}$. Thus, we obtain 
    \begin{equation}\label{eq:alg2_Eun_lower}
        E(u^n) - \zeta \ge \Delta_{n+1} - \tilde{K}(\Delta_{n-1} - \Delta_{n+1}).
    \end{equation}

    Based on \eqref{eq:alg2_Eun_lower}, we branch the analysis similarly by evaluating the expanded interval $(\Delta_{n-1} - \Delta_{n+1})$:

    \textbf{Case 1:} $\tilde{K}(\Delta_{n-1} - \Delta_{n+1}) \le \frac{1}{2}\Delta_{n+1}$. 
    In this case, \eqref{eq:alg2_Eun_lower} implies $E(u^n) - \zeta \ge \frac{1}{2}\Delta_{n+1} > 0$. Substituting this into \eqref{eq:alg2_kl_gap}, we obtain:
    \begin{equation}\label{eq:alg2_case1_descent}
        \Delta_{n-1} - \Delta_{n+1} \ge \frac{1}{C_1} \left( \frac{1}{2}\Delta_{n+1} \right)^{2\theta} := \gamma \Delta_{n+1}^{2\theta}.
    \end{equation}

    \textbf{Case 2:} $\tilde{K}(\Delta_{n-1} - \Delta_{n+1}) > \frac{1}{2}\Delta_{n+1}$. 
    This condition directly yields a linear decrease across two steps:
    \begin{equation}\label{eq:alg2_case2_descent}
        \Delta_{n+1} < \frac{2\tilde{K}}{2\tilde{K}+1} \Delta_{n-1} := q_2 \Delta_{n-1}, \quad \text{where } q_2 < 1.
    \end{equation}

    We proceed to analyze the finite length property based on the KL exponent $\theta \in [0, 1)$. 

    (i) When $\theta=0$:
By Lemma~\ref{lem:subgrad_bound_alg2_smooth_g2},
we have
\[
\dist(\mathbf{0},\partial E(u^n))
\le C\bigl(
\|u^n-x^n\|+\|x^n-x^{n-1}\|
\bigr)\to 0.
\]
Since the KL inequality with exponent $\theta=0$
provides a positive lower bound on
$\dist(\mathbf{0},\partial E(u^n))$
whenever $E(u^n)>\zeta$ in the relevant KL
neighborhood, we must have $E(u^n)\le\zeta$
for all sufficiently large $n\ge N_1$.
Using Lemma~\ref{lem:energy_relation_smooth_g2},
we obtain
\[
\Delta_{n+1}
\le (1-\lambda_n)\Delta_n
+\frac{\ell_{g_2,n}}{2}
\lambda_n(1-\lambda_n)\|u^n-x^n\|^2.
\]
By \eqref{eq:alg2_suff_descent},
\[
\|u^n-x^n\|^2
\le \frac{\Delta_n-\Delta_{n+1}}{M}.
\]
Consequently,
\[
\Delta_{n+1}
\le (1-\lambda_n)\Delta_n
+a_n(\Delta_n-\Delta_{n+1}),
\]
where
$a_n:=
\frac{\ell_{g_2,n}\lambda_n(1-\lambda_n)}{2M}
\le \frac{L_{g_2,\Omega}}{2M}$.
Rearranging the terms, we have
\[
\Delta_{n+1}
\le
\left(1-\frac{\lambda_n}{1+a_n}\right)\Delta_n
\le q\Delta_n,
\]
where
$q:=1-\frac{\lambda_{\min}}{1+\frac{L_{g_2,\Omega}}{2M}}\in[0,1)$.
Thus, for all $n\ge N_1$,
\[
\Delta_n\le\Delta_{N_1}q^{n-N_1}.
\]

Finally, since
$x^{n+1}-x^n=\lambda_n(u^n-x^n)$,
we have
\[
\begin{aligned}
\sum_{n=N_1}^{\infty}\|x^{n+1}-x^n\|
\le \frac{\lambda_0}{\sqrt M}
\sum_{n=N_1}^{\infty}
\sqrt{\Delta_n-\Delta_{n+1}}
\le \frac{\lambda_0\sqrt{\Delta_{N_1}}}{\sqrt M}
\sum_{j=0}^{\infty}q^{j/2}
<\infty.
\end{aligned}
\]
Including the initial iterations,
we conclude that
$\sum_{n=0}^{\infty}\|x^{n+1}-x^n\|<\infty$.

    (ii) When $\theta \in (0, 1/2]$: If Case 1 occurs, since $2\theta \le 1$ and $\Delta_{n+1} \to 0$, \eqref{eq:alg2_case1_descent} gives $\Delta_{n-1} - \Delta_{n+1} \ge \gamma \Delta_{n+1}$, yielding $\Delta_{n+1} \le q_1 \Delta_{n-1}$ with $q_1 = \frac{1}{1+\gamma} < 1$. Combined with Case 2, we always have $\Delta_{n+1} \le q \Delta_{n-1}$ where $q = \max\{q_1, q_2\} \in (0, 1)$. This naturally guarantees $\sum_{n=N_0}^{\infty} \sqrt{\Delta_{n-1} - \Delta_{n+1}} < \infty$.

    (iii) When $\theta \in (1/2, 1)$: Utilizing the exact same concavity inequality on $h(s) = cs^{1-\theta}$ as derived in Theorem \ref{thm:global_smooth_g2}, there exists a constant $C_3 > 0$ such that for all large $n$,
    \begin{equation*}
        \sqrt{\Delta_{n-1} - \Delta_{n+1}} \le C_3 (\Delta_{n-1}^{1-\theta} - \Delta_{n+1}^{1-\theta}).
    \end{equation*}
    Summing this  series from $n = N_0$ to $\infty$ yields $\sum_{n=N_0}^{\infty} \sqrt{\Delta_{n-1} - \Delta_{n+1}} \le C_3\left(\Delta_{N_0-1}^{1-\theta}+\Delta_{N_0}^{1-\theta}\right) < \infty$.

    Finally, across all $\theta \in [0, 1)$, we have established that $\sum_{n=N_0}^{\infty} \sqrt{\Delta_{n-1} - \Delta_{n+1}} < \infty$. Recalling $\Delta_n - \Delta_{n+1} \le \Delta_{n-1} - \Delta_{n+1}$ and applying the line search update $x^{n+1} - x^n = \lambda_n(u^n - x^n)$ along with \eqref{eq:alg2_suff_descent}, we conclude:
    \begin{align*}
        \sum_{n=N_0}^{\infty} \|x^{n+1} - x^n\| 
        &\le \lambda_0 \sum_{n=N_0}^{\infty} \|u^n - x^n\| 
        \le \frac{\lambda_0}{\sqrt{M}} \sum_{n=N_0}^{\infty} \sqrt{\Delta_n - \Delta_{n+1}} \\
        &\le \frac{\lambda_0}{\sqrt{M}} \sum_{n=N_0}^{\infty} \sqrt{\Delta_{n-1} - \Delta_{n+1}} < +\infty.
    \end{align*}
    Incorporating the finite initial steps, the entire sequence achieves a finite trajectory length $\sum_{n=0}^{\infty} \|x^{n+1} - x^n\| < \infty$, completing the proof.
\end{proof}

\section{Convergence Analysis of pABBDCA$_{\text{se}}$}\label{sec:alg3_conv}

To address the non-monotonicity introduced by the extrapolation, we must analyze the algorithm within a modified framework. We first establish a quasi-descent lemma regarding the original energy function $E(x)$, and subsequently introduce a Lyapunov function to restore monotonic descent, which is essential for the Kurdyka-\L ojasiewicz (KL) convergence analysis.

\begin{lemma}\label{lem:quasi_descent}
Let $u^n$ be the optimal solution of the proximal subproblem and $d^n = u^n - x^n$ be the search direction. For any accepted step size $\lambda_n \in (0, 1]$, the energy function satisfies
\begin{equation}\label{eq:E_sub}
    E(x^{n+1}) \le E(x^n) - \frac{2-\beta_n}{2\alpha_n\lambda_n}\|x^{n+1} - x^n\|^2 + \frac{\lambda_n\beta_n}{2\alpha_n}\|x^n - x^{n-1}\|^2.
\end{equation}
\end{lemma}
\begin{proof}
Let $F(x) = f(x) + g_1(x)$. With the strong convexity of the subproblem's objective $R_n(x):=\langle - \xi^n, x - y^n \rangle + \frac{1}{2\alpha_n}\|x - y^n\|^2 + f(x)+ g_1(x)$ with modulus $\frac{1}{\alpha_n}$, we have $R_n(x^n) \ge R_n(u^n) + \frac{1}{2\alpha_n}\|x^n - u^n\|^2$. Expanding this yields
\begin{equation*}
    F(u^n) - F(x^n) \le \langle \xi^n, u^n - x^n \rangle + \frac{1}{2\alpha_n} \left( \|x^n - y^n\|^2 - \|u^n - y^n\|^2 - \|u^n - x^n\|^2 \right).
\end{equation*}
By expanding $\|u^n - y^n\|^2 = \|(u^n - x^n) + (x^n - y^n)\|^2$ and recalling $x^n - y^n = -\beta_n(x^n - x^{n-1})$, we derive
\begin{equation}\label{eq:F_bound}
    F(u^n) - F(x^n) \le \langle \xi^n, d^n \rangle - \frac{1}{\alpha_n}\|d^n\|^2 + \frac{\beta_n}{\alpha_n}\langle d^n, x^n - x^{n-1} \rangle.
\end{equation}
For any $\lambda_n \in (0, 1]$, the convexity of $F(x)$ gives $F(x^n + \lambda_n d^n) \le F(x^n) + \lambda_n(F(u^n) - F(x^n))$. For the concave component $-g_2$, the subgradient inequality yields $-g_2(x^n + \lambda_n d^n) \le -g_2(x^n) - \lambda_n \langle \xi^n, d^n \rangle$. Summing these two bounds cancels the $\langle \xi^n, d^n \rangle$ term
\begin{equation*}
    E(x^n + \lambda_n d^n) \le E(x^n) - \frac{\lambda_n}{\alpha_n}\|d^n\|^2 + \frac{\lambda_n\beta_n}{\alpha_n}\langle d^n, x^n - x^{n-1} \rangle.
\end{equation*}
Applying the Cauchy-Schwarz and Young's inequality $\langle d^n, x^n - x^{n-1} \rangle \le \frac{1}{2}\|d^n\|^2 + \frac{1}{2}\|x^n - x^{n-1}\|^2$, and substituting $\|d^n\|^2 = \frac{1}{\lambda_n^2}\|x^{n+1} - x^n\|^2$, we arrive at the inequality \eqref{eq:E_sub}.
\end{proof}

\begin{lemma}\label{lem:alg3_direction_control}
The line search in Algorithm~\ref{alg:prox_bb_dca_ext_ada}
terminates after finitely many trials at each iteration.
Moreover, the search direction satisfies
\begin{equation}\label{eq:alg3_direction_control}
    \|d^n\|
    \le \|x^{n+1}-x^n\|
       +K_d\|x^n-x^{n-1}\|,
\end{equation}
where $K_d=\frac{1}{1-\eta\alpha_{\max}}$.
\end{lemma}

\begin{proof}
By the convexity of
$f+g_1$ and \eqref{eq:F_bound}, for every
$\lambda\in(0,1]$,
\begin{equation}\label{eq:alg3_trial_bound}
    E(x^n+\lambda d^n)
    \le E(x^n)
       -\frac{\lambda}{\alpha_n}\|d^n\|^2
       +\frac{\lambda\beta_n}{\alpha_n}
        \langle d^n,x^n-x^{n-1}\rangle.
\end{equation}
If $d^n=0$, the first trial is accepted.
Otherwise, $\nu_n>0$, and the line-search condition
holds whenever
\[
    \lambda\left[
        \left(\eta-\frac1{\alpha_n}\right)\|d^n\|^2
        +\frac{\beta_n}{\alpha_n}
         \langle d^n,x^n-x^{n-1}\rangle
    \right]\le\nu_n.
\]
This inequality holds for all sufficiently small
positive $\lambda$, proving finite termination.

If $\lambda_n=1$, then
$d^n=x^{n+1}-x^n$, so
\eqref{eq:alg3_direction_control} follows.
If $\lambda_n<1$, the first trial $\lambda=1$
was rejected. Combining its rejection with
\eqref{eq:alg3_trial_bound} gives
\[
    (1-\eta\alpha_n)\|d^n\|^2+\alpha_n\nu_n
    <\beta_n\langle d^n,x^n-x^{n-1}\rangle
    \le\|d^n\|\|x^n-x^{n-1}\|.
\]
Since $1-\eta\alpha_n\ge1-\eta\alpha_{\max}>0$,
we obtain $\|d^n\|\le K_d\|x^n-x^{n-1}\|$.
This proves \eqref{eq:alg3_direction_control}.
\end{proof}

To overcome the positive error term in \eqref{eq:E_sub}, we introduce a Lyapunov sequence. The restriction of $\beta_n$ in Algorithm \ref{alg:prox_bb_dca_ext_ada} is precisely designed to ensure its descent.

\begin{lemma}\label{lem:lyapunov_descent}
Suppose the sequence $\{x^n\}_n$ is generated by Algorithm \ref{alg:prox_bb_dca_ext_ada}. Define
\begin{equation}\label{eq:lyapunov}
    \mathcal{L}_n(x^n, x^{n-1}) := E(x^n) + \frac{\lambda_n\beta_n}{2\alpha_n}\|x^n - x^{n-1}\|^2
\end{equation}
as the Lyapunov function. Then, the sequence $\{\mathcal{L}_n(x^n, x^{n-1})\}_n$ is monotonically decreasing and
\begin{equation}\label{eq:sum_finite}
    \sum_{n=1}^{\infty} \|x^{n+1} - x^n\|^2 < \infty.
\end{equation}
\end{lemma}
\begin{proof}
With upper bound of $E(x^{n+1})$ from \eqref{eq:E_sub}, we obtain
\begin{align*}
    \mathcal{L}_{n+1}(x^{n+1}, x^n) \le &E(x^n) - \frac{2-\beta_n}{2\alpha_n\lambda_n}\|x^{n+1} - x^n\|^2 + \frac{\lambda_n\beta_n}{2\alpha_n}\|x^n - x^{n-1}\|^2 \\&+ \frac{\lambda_{n+1}\beta_{n+1}}{2\alpha_{n+1}}\|x^{n+1} - x^n\|^2 \\
    = &\mathcal{L}_n(x^n, x^{n-1}) - \left( \frac{2-\beta_n}{2\alpha_n\lambda_n} - \frac{\lambda_{n+1}\beta_{n+1}}{2\alpha_{n+1}} \right) \|x^{n+1} - x^n\|^2.
\end{align*}
According to the update rule \eqref{betan} in Algorithm \ref{alg:prox_bb_dca_ext_ada}, we consider two cases for $\beta_{n+1}$:

\textbf{Case 1:} If $\frac{\alpha_{n+1}}{\alpha_n\lambda_n}(2-\beta_n) - 2c_0\alpha_{n+1} \ge 0$, the update rule yields
\begin{equation*}
    \beta_{n+1} \le \frac{\alpha_{n+1}}{\alpha_n\lambda_n}(2-\beta_n) - 2c_0\alpha_{n+1} \implies \frac{2-\beta_n}{2\alpha_n\lambda_n} - \frac{\lambda_{n+1}\beta_{n+1}}{2\alpha_{n+1}} \ge \frac{2-\beta_n}{2\alpha_n\lambda_n} - \frac{\beta_{n+1}}{2\alpha_{n+1}} \ge c_0.
\end{equation*}

\textbf{Case 2:} If $\frac{\alpha_{n+1}}{\alpha_n\lambda_n}(2-\beta_n) - 2c_0\alpha_{n+1} < 0$, the truncation operation $\max\{0, \cdot\}$ forces $\beta_{n+1} = 0$. In this case, the negative term vanishes ($\frac{\lambda_{n+1}\beta_{n+1}}{2\alpha_{n+1}} = 0$), which brings
\begin{equation*}
    \frac{2-\beta_n}{2\alpha_n\lambda_n} - \frac{\lambda_{n+1}\beta_{n+1}}{2\alpha_{n+1}}=\frac{2-\beta_n}{2\alpha_n\lambda_n} - 0 \ge \frac{1}{2\alpha_{\max}}.
\end{equation*}

Combining both cases, and defining a global constant $\tilde{c} = \min\left\{c_0, \frac{1}{2\alpha_{\max}}\right\} > 0$, we obtain the monotonic descent of the Lyapunov function as
\begin{equation}\label{eq:L_strict_descent}
    \mathcal{L}_{n+1}(x^{n+1}, x^n) \le \mathcal{L}_n(x^n, x^{n-1}) - \tilde{c} \|x^{n+1} - x^n\|^2.
\end{equation}
Since $E(x)$ is bounded from below and $\beta_n \ge 0$, $\mathcal{L}_n(x^n, x^{n-1})$ is also bounded from below. Summing \eqref{eq:L_strict_descent} from $n=1$ to $N$ and letting $N \to \infty$ directly yields \eqref{eq:sum_finite}. 
\end{proof}

The Lyapunov descent and the direction estimate above
yield the following boundedness result.

\begin{corollary}\label{cor:boundedness_alg3}
The sequences $\{x^n\}_n$, $\{y^n\}_n$, and
$\{u^n\}_n$ generated by
Algorithm~\ref{alg:prox_bb_dca_ext_ada} are bounded,
and $\|u^n-x^n\|\to0$.
Consequently, they lie in a convex compact set
$\Omega$, on which $\nabla f$ is Lipschitz continuous
with constant $L_\Omega$.
Under Assumption~\ref{assum:smooth_g1} or
Assumption~\ref{assum:smooth_g2}, the corresponding
gradient $\nabla g_1$ or $\nabla g_2$ is Lipschitz
continuous on $\Omega$, with constant
$L_{g_1,\Omega}$ or $L_{g_2,\Omega}$, respectively.
\end{corollary}

\begin{proof}
By Lemma~\ref{lem:lyapunov_descent},
\[
    E(x^n)\le\mathcal L_n\le\mathcal L_0.
\]
The compactness of the sublevel sets of $E$
therefore implies the boundedness of $\{x^n\}_n$.
Moreover, \eqref{eq:sum_finite} and
Lemma~\ref{lem:alg3_direction_control} give
\[
    \|x^{n+1}-x^n\|\to0,
    \qquad
    \|u^n-x^n\|\to0.
\]
Thus $\{u^n\}_n$ and
$\{y^n=x^n+\beta_n(x^n-x^{n-1})\}_n$
are also bounded.
The uniform Lipschitz estimates on a convex
compact set follow from the proof of Proposition~\ref{prop:boundedness}.
\end{proof}

\subsection{Global convergence under Assumption \ref{assum:smooth_g1}}
Now we establish a relative subgradient bound for the fixed conjugate auxiliary function $\mathcal{E}(x, \xi)$, evaluated at the sequence $(x^n, \xi^n)$.

\begin{lemma}\label{lem:subgrad_alg3_smooth_g1}
    Under Assumption \ref{assum:smooth_g1}, let $\{x^n\}_n$ and $\{\xi^n\}_n$ be the sequences generated by Algorithm \ref{alg:prox_bb_dca_ext_ada}. Then, there exists a constant $C > 0$ such that for all $n \ge 1$:
    \begin{equation}\label{eq:subgrad_E_alg3}
        \dist(\emph{\textbf{0}}, \partial \mathcal{E}(x^n, \xi^n)) \le C \left( \|x^{n+1} - x^n\| + \|x^n - x^{n-1}\| \right).
    \end{equation}
\end{lemma}

\begin{proof}
    Since $\xi^n \in \partial g_2(x^n)$, the Fenchel-Young equality yields $g_2(x^n) = \langle x^n, \xi^n \rangle - g_2^*(\xi^n)$ and $x^n \in \partial g_2^*(\xi^n)$, which implies $0 \in -x^n + \partial g_2^*(\xi^n)$. Moreover, the first-order optimality condition of the proximal subproblem for $u^n$ yields
    \begin{align*}
        \nabla f(u^n) + \nabla g_1(u^n) - \xi^n + \frac{1}{\alpha_n}(u^n - y^n) = 0,
    \end{align*}
    where $y^n = x^n + \beta_n(x^n - x^{n-1})$. Substituting $\xi^n$ into $\partial \mathcal{E}(x^n, \xi^n)$, we obtain
    \begin{align*}
        \begin{pmatrix}
            (\nabla f(x^n) - \nabla f(u^n)) + (\nabla g_1(x^n) - \nabla g_1(u^n)) - \frac{1}{\alpha_n}(u^n - y^n) \\
            0
        \end{pmatrix} \in \partial \mathcal{E}(x^n, \xi^n).
    \end{align*}
       Consequently, setting
        $A=L_\Omega+L_{g_1,\Omega}
          +\frac1{\alpha_{\min}},
        B=\frac{1}{\alpha_{\min}}$,
    we obtain
    \begin{align*}
        \dist(\textbf{0},
             \partial\mathcal E(x^n,\xi^n))
        &\le A\|u^n-x^n\|
             +B\|x^n-x^{n-1}\| \\
        &\le A\|x^{n+1}-x^n\|
             +(AK_d+B)\|x^n-x^{n-1}\| \\
        &\le C\left(
             \|x^{n+1}-x^n\|
             +\|x^n-x^{n-1}\|
             \right),
    \end{align*}
    where $C=\max\{A,AK_d+B\}$ and  the second inequality follows from Lemma~\ref{lem:alg3_direction_control}.
\end{proof}

With the monotonic descent of the Lyapunov sequence $\mathcal{L}_n(x^n, x^{n-1})$ and the relative subgradient bound on the fixed auxiliary function $\mathcal{E}(x, \xi)$, we now establish the global convergence of pABBDCA$_{\text{se}}$.

\begin{theorem}\label{thm:global_alg3_smooth_g1}
    Under Assumption \ref{assum:smooth_g1}, suppose $\mathcal{E}(x, \xi)$ is a KL function with exponent $\theta \in [0, 1)$. Then, the sequence $\{x^n\}_n$ generated by Algorithm \ref{alg:prox_bb_dca_ext_ada} globally converges to a critical point $x^*$ with a finite trajectory length, i.e.,
    \begin{equation*}
        \sum_{n=0}^{\infty} \|x^{n+1} - x^n\| < \infty.
    \end{equation*}
\end{theorem}

\begin{proof}
    Set $a_n=\|x^{n+1}-x^n\|$ for $n\ge0$, $a_{-1}=0$ and
    \[
        \delta_n=\frac{\lambda_n\beta_n}{2\alpha_n},
        \qquad
        0\le\delta_n\le
        \frac{\lambda_0\overline{\beta}}{2\alpha_{\min}}
        =:\delta_{\max}.
    \]
    By Lemma \ref{lem:lyapunov_descent}, the sequence
    $\mathcal{L}_n=E(x^n)+\delta_na_{n-1}^2$
    decreases to some $\zeta\in\mathbb{R}$. Writing
    $\Delta_n=\mathcal{L}_n-\zeta$, we have
    \begin{equation}\label{eq:alg3_g1_descent}
        \Delta_n-\Delta_{n+1}\ge\tilde c\,a_n^2
    \end{equation}
    for some $\tilde c>0$. Thus $a_n\to0$, and 
    \[
        \mathcal{E}(x^n,\xi^n)=E(x^n)
        =\mathcal{L}_n-\delta_na_{n-1}^2\to\zeta.
    \]

    By Corollary \ref{cor:boundedness_alg3}, $\{x^n\}_n$
    is bounded. The subdifferential formula for $\mathcal E$
    and Lemma \ref{lem:subgrad_alg3_smooth_g1} imply
    \[
        \|\nabla f(x^n)+\nabla g_1(x^n)-\xi^n\|
        \le C(a_n+a_{n-1})\to0.
    \]
    Hence $\{\xi^n\}_n$ is bounded as well. Let
    $\Omega_{\mathcal E}$ be the cluster set of
    $\{(x^n,\xi^n)\}_n$. It is nonempty and compact, with
    $\dist((x^n,\xi^n),\Omega_{\mathcal E})\to0$.

    For any subsequence
    $(x^{n_j},\xi^{n_j})\to(\bar x,\bar\xi)$,
    the closedness of $\partial g_2$ yields
    \[
        \bar\xi\in\partial g_2(\bar x),
        \qquad
        \bar\xi=\nabla f(\bar x)+\nabla g_1(\bar x).
    \]
    Moreover,
    \[
        g_2(x^{n_j})\le g_2(\bar x)
        +\langle\xi^{n_j},x^{n_j}-\bar x\rangle.
    \]
    Together with lower semicontinuity, this gives
    $g_2(x^{n_j})\to g_2(\bar x)$. Consequently,
    $\mathcal E(\bar x,\bar\xi)=E(\bar x)=\zeta$ and
    $\textbf{0}\in\partial\mathcal E(\bar x,\bar\xi)$.
    In particular, every such $\bar x$ is a critical point
    of the original problem.

    If $\Delta_n=0$ for some $n$, then
    \eqref{eq:alg3_g1_descent} forces the subsequent iterates
    to remain constant, and the conclusion follows.
    We therefore assume $\Delta_n>0$.

    Define the fixed function and auxiliary sequence
    \[
        \mathcal H(x,\xi,t):=\mathcal E(x,\xi)+t^2,
        \qquad
        t_n=\sqrt{\delta_n}\,a_{n-1},
        \qquad
        w^n=(x^n,\xi^n,t_n).
    \]
    Thus, we have
\[
    \mathcal H(w^n)
    =\mathcal E(x^n,\xi^n)+t_n^2
    =E(x^n)+\delta_n\|x^n-x^{n-1}\|^2
    =\mathcal L_n.
\]
and
$\Delta_n=\mathcal H(w^n)-\zeta$.
This allows us to combine the descent of $\mathcal L_n$
with the KL inequality for $\mathcal H$.
    Since $\mathcal E\equiv\zeta$ on $\Omega_{\mathcal E}$,
    its KL property provides
    $\varepsilon,\eta_0>0$ and a concave
    function $\varphi$ such that
    \[
        \varphi'(\mathcal E(x,\xi)-\zeta)
        \dist(\textbf{0},\partial\mathcal E(x,\xi))\ge1
    \]
    whenever $\dist((x,\xi),\Omega_{\mathcal E})<\varepsilon$
    and $0<\mathcal E(x,\xi)-\zeta<\eta_0$.

    Motivated by the construction for separable sums in
\cite[Theorem~3.8]{Wang2023}, we set
$\psi(s)=2\varphi(s/2)+\sqrt{2s}$ on $[0,\eta_0)$.
    This function satisfies the KL conditions, and
    \[
        \psi'(s)=\varphi'(s/2)+(2s)^{-1/2},
        \qquad
        \partial\mathcal H(x,\xi,t)
        =\partial\mathcal E(x,\xi)\times\{2t\}.
    \]
    For $\dist((x,\xi),\Omega_{\mathcal E})<\varepsilon$,
    write $s=\mathcal H(x,\xi,t)-\zeta\in(0,\eta_0)$ and
    $r=\mathcal E(x,\xi)-\zeta$, so that $s=r+t^2$.
    If $r\ge s/2$, concavity gives
    $\varphi'(s/2)\ge\varphi'(r)$, and hence
    \[
        \psi'(s)\dist(\textbf{0},\partial\mathcal H(x,\xi,t))
        \ge\varphi'(r)
        \dist(\textbf{0},\partial\mathcal E(x,\xi))\ge1.
    \]
    If $r<s/2$, then $t^2>s/2$, so
    \[
        \psi'(s)\dist(\textbf{0},\partial\mathcal H(x,\xi,t))
        \ge\frac{2|t|}{\sqrt{2s}}>1.
    \]
    Since $\Delta_n\to0$ and
    $\dist((x^n,\xi^n),\Omega_{\mathcal E})\to0$,
    it follows that, for all $n\ge N_0$ sufficiently large,
    \begin{equation}\label{eq:alg3_g1_augmented_kl}
        \psi'(\Delta_n)
        \dist(\textbf{0},\partial\mathcal H(w^n))\ge1.
    \end{equation}

    On the other hand, with
    $\widetilde C=C+2\sqrt{\delta_{\max}}$, we have
    \[
        \dist(\textbf{0},\partial\mathcal H(w^n))
        \le C(a_n+a_{n-1})+2t_n
        \le\widetilde C(a_n+a_{n-1}).
    \]
    Let $\phi_n=\psi(\Delta_n)-\psi(\Delta_{n+1})$.
    Concavity, \eqref{eq:alg3_g1_augmented_kl}, and
    \eqref{eq:alg3_g1_descent} yield
    \[
        \phi_n\widetilde C(a_n+a_{n-1})
        \ge\Delta_n-\Delta_{n+1}
        \ge\tilde c\,a_n^2.
    \]
    Thus, setting $\widetilde K=\widetilde C/\tilde c$, we obtain
    \[
        a_n\le
        \sqrt{\widetilde K\phi_n(a_n+a_{n-1})}
        \le\widetilde K\phi_n+\frac14(a_n+a_{n-1}).
    \]
    Rearranging and summing from $n=N_0$ to $N$ gives
    \[
        \frac12\sum_{n=N_0}^{N}a_n
        \le\widetilde K
           [\psi(\Delta_{N_0})-\psi(\Delta_{N+1})]
           +\frac14(a_{N_0-1}-a_N)
        \le\widetilde K\psi(\Delta_{N_0})
           +\frac14a_{N_0-1}.
    \]
    Letting $N\to\infty$ proves
    $\sum_{n=0}^{\infty}\|x^{n+1}-x^n\|<\infty$.
    Hence $x^n$ converges to some $x^*$, which is critical
    by the cluster-point characterization above.
\end{proof}

\subsection{Global convergence under Assumption \ref{assum:smooth_g2}}

In this section, we establish the global convergence of Algorithm \ref{alg:prox_bb_dca_ext_ada} (pABBDCA$_{\text{se}}$) under Assumption \ref{assum:smooth_g2}, where $g_2$ is smooth and $g_1$ is allowed to be nonsmooth. Since the trajectory is confined within the compact set $\Omega$, the gradient $\nabla g_2$ is uniformly $L_{g_2, \Omega}$-Lipschitz continuous over $\Omega$. As discussed in Section~\ref{sec:pb_nonsmooth_g1}, the proximal optimality condition provides a subgradient estimate at $u^n$, rather than directly at $x^n$. Instead, we bridge the descent of $\mathcal{L}_n$ with the KL property evaluated at the subproblem minimizer $u^n$.

First, since Algorithm \ref{alg:prox_bb_dca_ext_res} and Algorithm \ref{alg:prox_bb_dca_ext_ada} share the identical proximal subproblem structure, the relative subgradient bound evaluated at $u^n$ remains the same.

\begin{lemma}\label{lem:subgrad_bound_alg3_nonsmooth}
    Under Assumption \ref{assum:smooth_g2}, let $\{x^n\}_n$ and $\{u^n\}_n$ be the sequences generated by Algorithm \ref{alg:prox_bb_dca_ext_ada}. There exists a constant $C > 0$ such that
    \begin{equation}
    {\dist}(\emph{\textbf{0}}, \partial E(u^n)) \le C \left( \|u^n - x^n\| + \|x^n - x^{n-1}\| \right),
    \end{equation}
    where $C = L_{g_2, \Omega} + \frac{1}{\alpha_{\min}} > 0$.
\end{lemma}

Now, we present the global convergence theorem by bridging the KL inequality evaluated at $u^n$ with the descent of the Lyapunov function $\mathcal{L}_n$.

\begin{theorem}\label{thm:global_alg3_nonsmooth}
    Under Assumption~\ref{assum:smooth_g2}, let $\{x^n\}_n$ be a bounded sequence generated by Algorithm~\ref{alg:prox_bb_dca_ext_ada}. Suppose $E(x)$ is a KL function with exponent $\theta \in [0, 1)$. Then, the sequence $\{x^n\}_n$ globally converges to a critical point $x^*$ with a finite trajectory length, i.e.,
    \begin{equation}
        \sum_{n=0}^{\infty} \|x^{n+1} - x^n\| < \infty.
    \end{equation}
\end{theorem}

\begin{proof}
Set $a_{-1}=0$, $a_n=\|x^{n+1}-x^n\|$ for $n\ge0$,
and
\[
    \delta_n=\frac{\lambda_n\beta_n}{2\alpha_n},
    \qquad
    0\le\delta_n\le
    \frac{\overline\beta}{2\alpha_{\min}}
    =:\delta_{\max}.
\]
From Lemma~\ref{lem:lyapunov_descent}, the sequence
$\mathcal L_n=E(x^n)+\delta_na_{n-1}^2$
monotonically decreases to some $\zeta\in\mathbb R$.
Let $\Delta_n=\mathcal L_n-\zeta$. Then
\begin{equation}\label{eq:alg3_suff_descent}
    \Delta_n-\Delta_{n+1}\ge\tilde c\,a_n^2.
\end{equation}
Consequently, $a_n\to0$ and
$E(x^n)=\mathcal L_n-\delta_na_{n-1}^2\to\zeta$.
By Lemma~\ref{lem:alg3_direction_control}, we also have
\[
    \|d^n\|\le a_n+K_da_{n-1}\to0.
\]
For $n\ge1$, define
\[
    D_n:=\Delta_{n-1}-\Delta_{n+1}
    \ge\tilde c(a_n^2+a_{n-1}^2).
\]

If $\lambda_n=1$, then $u^n=x^{n+1}$ and
\[
    E(u^n)-\zeta
    =\Delta_{n+1}-\delta_{n+1}a_n^2
    \ge\Delta_{n+1}-\delta_{\max}a_n^2.
\]
If $\lambda_n<1$, the first trial is rejected.
Using $\|d^n\|\le K_da_{n-1}$ from the proof of
Lemma~\ref{lem:alg3_direction_control}, we obtain
\begin{align*}
    E(u^n)-\zeta
    &>\Delta_n-\delta_na_{n-1}^2
      -\eta\|d^n\|^2+\nu_n\\
    &\ge\Delta_{n+1}
      -(\delta_{\max}+\eta K_d^2)a_{n-1}^2.
\end{align*}
Combining both cases yields
\begin{equation}\label{eq:alg3_bridge}
    E(u^n)-\zeta
    \ge\Delta_{n+1}-\widehat K D_n,
\end{equation}
where
$\widehat K=(\delta_{\max}+\eta K_d^2)/\tilde c>0$.
On the other hand, evaluating
\eqref{eq:alg3_trial_bound} at $u^n=x^n+d^n$ gives
\[
    E(u^n)-E(x^n)
    \le-\frac{\|d^n\|^2}{\alpha_n}
       +\frac{\beta_n}{\alpha_n}
        \langle d^n,x^n-x^{n-1}\rangle.
\]
Since $d^n\to0$, $a_{n-1}\to0$, and
$E(x^n)\to\zeta$, we get
$\limsup_n E(u^n)\le\zeta$.
Together with \eqref{eq:alg3_bridge}, which gives
$\liminf_n E(u^n)\ge\zeta$, this implies
$E(u^n)\to\zeta$.
By Corollary~\ref{cor:boundedness_alg3},
$\{u^n\}_n$ is bounded.
From the optimality condition of the proximal subproblem,
we have
\[
    v^n:=
    \nabla g_2(x^n)-\nabla f(u^n)
    -\frac1{\alpha_n}(u^n-y^n)
    \in\partial g_1(u^n).
\]
For any subsequence $u^{n_j}\to\bar x$,
we have $x^{n_j}\to\bar x$ and
$u^{n_j}-y^{n_j}\to0$. Thus
\[
    v^{n_j}\to
    \nabla g_2(\bar x)-\nabla f(\bar x)
    \in\partial g_1(\bar x),
\]
where the closedness of $\partial g_1$ was used.
This indicates that $\bar x$ is a critical point.
Moreover, the subgradient inequality gives
\[
    g_1(u^{n_j})
    \le g_1(\bar x)
       +\langle v^{n_j},u^{n_j}-\bar x\rangle.
\]
Together with lower semicontinuity, this yields
$g_1(u^{n_j})\to g_1(\bar x)$ and hence
$E(\bar x)=\zeta$.

Therefore, $E$ equals $\zeta$ on the compact cluster set
of $\{u^n\}_n$.
By Lemmas~\ref{lem:subgrad_bound_alg3_nonsmooth}
and~\ref{lem:alg3_direction_control}, there exists
$C>0$ such that
\[
    \dist(\textbf{0},\partial E(u^n))
    \le C(a_n+a_{n-1})\to0.
\]
Using the uniformized KL property, for all sufficiently
large $n$ with $E(u^n)>\zeta$, we obtain
\begin{equation}\label{eq:alg3_kl_gap}
\begin{aligned}
    (E(u^n)-\zeta)^{2\theta}
    \le c_{\rm KL}^{-2}
         \dist(\textbf{0},\partial E(u^n))^2
    \le\frac{2C^2}{c_{\rm KL}^2}
         (a_n^2+a_{n-1}^2)
    \le C_1D_n,
\end{aligned}
\end{equation}
where $c_{\rm KL}>0$ is the KL constant and
$C_1=2C^2/(c_{\rm KL}^2\tilde c)$.

If $\Delta_N=0$ for some $N$, then
\eqref{eq:alg3_suff_descent} implies that the subsequent
iterates remain constant, and the conclusion follows.
We therefore assume $\Delta_n>0$.

When $\theta=0$, the KL inequality and
$\dist(\textbf{0},\partial E(u^n))\to0$
imply $E(u^n)\le\zeta$ for all sufficiently large $n$.
From \eqref{eq:alg3_bridge}, we get
\[
    \Delta_{n+1}\le q\Delta_{n-1},
    \qquad
    q=\frac{\widehat K}{1+\widehat K}\in(0,1).
\]
Thus both parity subsequences decay geometrically, giving
\[
    \sum_{n=N_0}^{\infty}\sqrt{D_n}
    \le\sum_{n=N_0}^{\infty}\sqrt{\Delta_{n-1}}
    <\infty
\]
for some sufficiently large $N_0$.

When $\theta\in(0,1)$, set $r=\max\{\theta,1/2\}<1$.
Choose $N_0$ large enough that the preceding KL estimates
hold and $\Delta_{n-1}\le1$ for $n\ge N_0$.
If $\widehat K D_n\le\Delta_{n+1}/2$, then
\eqref{eq:alg3_bridge} gives
$E(u^n)-\zeta\ge\Delta_{n+1}/2>0$.
Substituting this into \eqref{eq:alg3_kl_gap}, we obtain
\[
    D_n\ge\frac{2^{-2\theta}}{C_1}
               \Delta_{n+1}^{2\theta}.
\]
Otherwise, $D_n>\Delta_{n+1}/(2\widehat K)$.
Since $\Delta_{n+1}\le1$, both cases yield
\[
    D_n\ge\gamma\Delta_{n+1}^{2r},
    \qquad
    \gamma:=\min\left\{
        \frac{2^{-2\theta}}{C_1},
        \frac1{2\widehat K}
    \right\}>0.
\]

If $\Delta_{n+1}\ge\Delta_{n-1}/2$, concavity gives
\begin{align*}
    \Delta_{n-1}^{1-r}-\Delta_{n+1}^{1-r}
    \ge(1-r)\Delta_{n-1}^{-r}D_n
    \ge(1-r)2^{-r}\sqrt{\gamma}\sqrt{D_n}.
\end{align*}
If $\Delta_{n+1}<\Delta_{n-1}/2$, then
\begin{align*}
    \Delta_{n-1}^{1-r}-\Delta_{n+1}^{1-r}
    \ge(1-2^{r-1})\Delta_{n-1}^{1-r}
    \ge(1-2^{r-1})\sqrt{D_n},
\end{align*}
where $r\ge1/2$ and $\Delta_{n-1}\le1$ were used.
Consequently, there exists $C_2>0$ such that
\[
    \sqrt{D_n}\le C_2
    \bigl(\Delta_{n-1}^{1-r}
          -\Delta_{n+1}^{1-r}\bigr).
\]
Summing from $N_0$ to $N$, we have
\[
    \sum_{n=N_0}^{N}\sqrt{D_n}
    \le C_2\bigl(
        \Delta_{N_0-1}^{1-r}
        +\Delta_{N_0}^{1-r}
    \bigr).
\]
Letting $N\to\infty$ gives
$\sum_{n=N_0}^{\infty}\sqrt{D_n}<\infty$.

Finally, for all $\theta\in[0,1)$,
\eqref{eq:alg3_suff_descent} yields
\[
    \sum_{n=N_0}^{\infty}\|x^{n+1}-x^n\|
    \le\frac1{\sqrt{\tilde c}}
       \sum_{n=N_0}^{\infty}\sqrt{D_n}
    <\infty.
\]
Including the finite initial steps, we conclude that
$\sum_{n=0}^{\infty}\|x^{n+1}-x^n\|<\infty$.
\end{proof}
\begin{remark}
\label{rem:segmentwise_uniform_lip}
In the preceding analysis, the segmentwise moduli
$\ell_{f,n}$, $\ell_{g_1,n}$, and $\ell_{g_2,n}$ are retained
whenever a single-iteration estimate is derived, so that the
corresponding bounds reflect the local behavior along
$\Gamma_n=\text{conv}\{x^n,u^n\}$.
When an iteration-independent constant is required in the KL
and summability arguments, we use the uniform upper bounds on
the compact set $\Omega$ established above.
This uniformization is introduced only for analytical convenience
and does not require the algorithms to compute any Lipschitz
constants.
\end{remark}

\section{Numerical experiments} \label{sec:pre}
In this section, we perform two numerical experiments to demonstrate the efficiency of our algorithm in solving nonconvex optimization problems:

\textbf{(1) Least squares problems with SCAD regularizer}: conducted on a computer with Intel(R) Core(TM) Ultra 5 125H (3.60 GHz). 

\textbf{(2) Graph Ginzburg-Landau model}: executed on a workstation with Intel(R) Xeon(R) CPU E5-2699A v4 (2.40GHz) and a GPU of NVIDIA GeForce RTX 2080 Ti.

Implementation details, parameter settings, and comparative results will be discussed below. Additionally, Remark \ref{KL:two:models} illustrates the KL properties of the two problems.

\subsection{Least squares problems with SCAD regularizer}
 We consider the smoothly clipped absolute deviation (SCAD) regularization, whose DC decomposition can be expressed as (\cite[Section 6.1]{APX} or \cite{Wen2018})
 \begin{equation*}
    P(x) = \mu\sum_{i=1}^{k}\int_0^{|x_i|}\min\left\{1,\frac{[\theta\mu-t]_+}{(\theta-1)\mu}\right\} dt=\mu\|x\|_1 -  \underbrace{\mu\sum_{i=1}^{k}\int_0^{|x_i|}\frac{[\min\{\theta\mu,t\}-\mu]_+}{(\theta-1)\mu}dt}_{\tilde P(x)}
\end{equation*}
 where $\theta>2$ is a constant, $\mu>0$ serves as the regularization parameter, $[x]_{+}=\max\{0,x\}$ and $\tilde P(x) =  \mu \|x\|_1 -P(x) = \sum_{i=1}^k \tilde p_i(x_i)$. The component functions and the gradient of $\tilde P$ are given by
 \begin{align*}
    \tilde  p_i(x_i) &= \left\{
        \begin{array}{ll}
         0& \text{if}  \ \  \ |x_i|\leq\mu \\
         \frac{(|x_i| - \mu)^2}{2(\theta - 1)}& \text{if} \ \  \ \mu<|x_i|<\theta\mu \\
         \mu|x_i| - \frac{\mu^2(\theta+1)}{2}&\text{if}\ \ \ |x_i|\geq\theta\mu 
    \end{array},
        \right. \\ \nabla \tilde p_i(x_i)&= \text{sign}(x_i) \dfrac{[\min\{\theta \mu,|x_i|\} -\mu]_{+}} {(\theta-1)}.
\end{align*}
 
 Applying SCAD regularization to the least squares problem, we obtain the following formulation:
\begin{equation}\label{scad}
    \min_{x\in \mathbb{R}^k}E(x) = \frac12\|Ax-b\|^2 + P(x)= \frac12\|Ax-b\|^2+\mu\|x\|_1-\tilde P(x),  \quad A\in \mathbb{R}^{m\times k},  \ b\in \mathbb{R}^{m}.
\end{equation}
    
In our implementation, linearizing the smooth component $f(x) = \frac{1}{2}\|Ax-b\|^2$ reduces the proximal subproblem to an exact closed-form soft-thresholding operation. To formulate the least squares problem, we first generate an $m \times k$ random matrix $A$ with normalized columns ($\|A_j\|_2 = 1$ for all $j = 1,...,k$). We then construct a sparse vector $y \in \mathbb{R}^k$ by uniformly randomly selecting its support $T \subset \{1,...,k\}$ with size $|T| = s$. In other words, $s$ denotes the number of non-zero elements in $y$, which characterizes its sparsity level. The vector $b \in \mathbb{R}^m$ is generated according to
\begin{equation*}
    b = Ay +  0.01 \xi,\quad \xi \sim \mathcal{N}(0,I_m),
\end{equation*}
where $\xi$ consists of i.i.d. standard Gaussian entries. All algorithms are initialized at the origin and terminate when the {\it relative step length} satisfies
\begin{equation*}
\text{RSL}(x^n):= \frac{\|x^{n} - x^{n-1}\|}{\max\{1, \|x^{n}\|\}} < 10^{-12}.
\end{equation*}
Here $\text{RSL}(x^n)$ denotes the relative step length \cite[equation (5.2)]{Wen2018}. We focus on problem sizes characterized by the tuple \( (m, k, s) = (720i, 2560i, 80i) \), where  \( i \) is an integer ranging from 1 to 10. This setup aligns with the high-dimensional setting in sparse statistics where the data dimension exceeds the number of data points, and the solution is sufficiently sparse. 

In our experiments, we compute SCAD \eqref{scad} with different algorithms for comparative analysis. In the evaluation of various algorithms, the iteration count (denoted as \( \text{iter} \)) and the CPU time are recorded. Concurrently, we record the number of nonzero elements in the output vector to assess the sparsity of the computed solution.

The computational results are summarized in Table~\ref{table1}, corresponding to our problem~\eqref{scad} with parameters $\mu = 0.033$ and $\theta = 10$. For each problem size, five random instances are generated, and the values reported in each row of Table~\ref{table1} represent the average performance across all instances. To evaluate the performance of our proposed pBBDCA and its two extrapolation variants, we compare it against four established algorithms: the classical DC Algorithm (DCA), the Boosted DC Algorithm with backtracking (BDCA), the proximal DC algorithm with extrapolation (pDCA$_e$) and the proximal DC algorithm with nonmonotone line search (EPDCA). The implementation details of these algorithms are discussed below.

\begin{itemize}
    \item \textbf{pBBDCA}: This algorithm is represented by Algorithm \ref{alg:prox_bb_dca} where $f(x)=\frac12\|Ax-b\|^2$, $g_1(x)=\mu\|x\|_1$ and $g_2(x) =  \tilde P(x)$. The parameter values for this algorithm are chosen as follows:  $\alpha_{\min} =10^{-4}$, $\alpha_{\max} =10^{4}$, $\eta=10^{-4}-10^{-15}$, $\omega = 0.9$ and the line search decay factor $\rho = 0.5$.

    \item \textbf{pABBDCA}: This algorithm is identical to pBBDCA, except that it replaces the standard BB1 step size with an alternating Barzilai-Borwein (ABB) step size strategy. All other parameter settings remain the same as in pBBDCA.

     \item \textbf{pABBDCA$_{\text{er}}$}: This algorithm corresponds to our Algorithm \ref{alg:prox_bb_dca_ext_res}, which incorporates an alternating BB step size and extrapolation with a restart mechanism. The convex splitting and the base parameter values for this algorithm remain consistent with pBBDCA. Additionally, the extrapolation upper bound is chosen as $\overline{\beta} = 0.9$, and the maximum line search iterations limit for triggering a restart is set to $N_{\max} = 5$.

    \item \textbf{pABBDCA$_{\text{se}}$}: This algorithm is presented by our Algorithm \ref{alg:prox_bb_dca_ext_ada}, which employs an adaptive safeguarded extrapolation strategy. It shares the identical convex splitting, basic parameter settings, and the extrapolation upper bound with pABBDCA$_{\text{er}}$. The positive constant for the safeguarded condition is set to $c_0 = 10^{-7}$.

     \item \textbf{EPDCA}: This algorithm is represented by \cite{luzhaosong}. It repeatedly solves the subproblems by updating the BB stepsize through a different type of line search. The parameter settings are: $\alpha_{\min} = 10^{-4}$, $  \alpha_{\max} = 10^{4}  $, $ \eta=10^{-4}-10^{-15}$, $  \omega = 0.9  $, line search decay factor $  \rho = 0.5 $, nonmonotone memory window $ M=10$, and sufficient decrease parameter $  \gamma = 10^{-4}  $.
    
   \item \textbf{pDCA$_e$}: Proposed by \cite{Wen2018}, this algorithm accelerates the standard proximal DC framework by incorporating a Nesterov-type extrapolation technique. The convex splitting remains consistent with our framework. The extrapolation parameter $\beta_k$ is updated and restarted adaptively.

    \item \textbf{BDCA}: This algorithm is based on a combination of DCA together with an aggressive line search technique, as detailed in \cite[Algorithm 2]{Artacho2018}. The parameter values of the line search part for this algorithm are chosen as follows:  $\lambda_{\max} = 10$,  $\alpha = 10^{-4}$, $\beta = 0.6$.
    \item \textbf{DCA}: This is the classical DC algorithm, which can be regarded as a special version of the algorithm BDCA without the line search \cite[Algorithm 1]{Artacho2018}. 
\end{itemize}

Table \ref{table1} summarizes the computational results for solving the SCAD regularization problem across randomly generated instances of increasing dimensions. It is evident that our pBBDCA algorithm significantly outperforms all compared methods in terms of both CPU time and iteration count. Furthermore, pBBDCA consistently requires the fewest iterations across all problem scales. This algorithmic efficiency translates into an advantage in CPU time, which becomes increasingly pronounced as the problem dimension grows. For the largest problem size ($m=7200, k=25600$), pBBDCA requires only 3.0 seconds, making it more than twice as fast as EPDCA (6.7s) and an order of magnitude faster than classical DCA (32.7s). Here, $s'$ denotes the number of nonzero entries in the solution
returned by pBBDCA. The accurately recovered sparsity ($s'$) remains highly consistent with the ground truth ($s$).

\begin{table}[htbp!]
\centering
\caption[Solving \eqref{scad} on random instances]{Solving \eqref{scad} on random instances\\(\protect\textcircled{1}\textup{DCA} \protect\textcircled{2}\textup{BDCA} \protect\textcircled{3}\textup{pDCA}$_{\mathrm{e}}$ \protect\textcircled{4}\textup{EPDCA} \protect\textcircled{5}\textup{pBBDCA})}\label{table1}
{
\begin{tabular}{ccc|rrrrr|rrrrr|c}
\hline
\multicolumn{3}{c|}{Size} & \multicolumn{5}{c|}{iter} & \multicolumn{5}{c|}{CPU time (s)} & \multicolumn{1}{c}{Sparsity}\\ 
\hline
$m$ & $k$ & $s$ & \textcircled{1} & \textcircled{2} & \textcircled{3} & \textcircled{4} & \textcircled{5} & \textcircled{1} & \textcircled{2} & \textcircled{3} & \textcircled{4} & \textcircled{5} &$s'$ \\ \hline
720 & 2560 & 80 & 577 & 111 & 155 & 57 & \textbf{52} & 0.3 & 0.4 & \textbf{0.1} & \textbf{0.1} & \textbf{0.1} & 81 \\
1440 & 5120 & 160 & 581 & 114 & 155 & 58 & \textbf{54} & 1.3 & 1.7 & 0.4 & 0.3 & \textbf{0.1} & 159 \\
2160 & 7680 & 240 & 584 & 112 & 156 & 60 & \textbf{53} & 3.4 & 3.3 & 1.0 & 0.8 & \textbf{0.4} & 242 \\
2880 & 10240 & 320 & 582 & 114 & 153 & 58 & \textbf{52} & 5.8 & 5.5 & 1.6 & 1.2 & \textbf{0.5} & 320 \\
3600 & 12800 & 400 & 595 & 120 & 149 & 58 & \textbf{54} & 8.6 & 8.4 & 2.2 & 1.7 & \textbf{0.8} & 403 \\
4320 & 15360 & 480 & 600 & 118 & 157 & 60 & \textbf{51} & 11.9 & 11.0 & 3.2 & 2.5 & \textbf{1.0} & 480 \\
5040 & 17920 & 560 & 588 & 116 & 156 & 58 & \textbf{52} & 14.9 & 14.0 & 4.0 & 3.0 & \textbf{1.3} & 562 \\
5760 & 20480 & 640 & 591 & 115 & 155 & 57 & \textbf{50} & 19.4 & 18.5 & 5.2 & 4.0 & \textbf{1.7} & 642 \\
6480 & 23040 & 720 & 599 & 114 & 155 & 58 & \textbf{54} & 24.7 & 23.4 & 6.5 & 5.0 & \textbf{2.3} & 721 \\
7200 & 25600 & 800 & 589 & 115 & 151 & 57 & \textbf{52} & 32.7 & 31.8 & 8.4 & 6.7 & \textbf{3.0} & 804 \\ \hline
\end{tabular}
}
\end{table}

To comprehensively demonstrate the numerical efficiency and scalability of the proposed pBBDCA, we present the convergence profiles for two representative problem scales: $i=5$ ($m=3600, k=12800$) and $i=10$ ($m=7200, k=25600$) in Figure \ref{fig:conv_index5} and Figure \ref{fig:conv_index10}, respectively. To provide a multi-dimensional evaluation, Figure \ref{fig:conv_index5} illustrates the evolution of the objective gap, $E(x^k) - E_{\min}$, whereas Figure \ref{fig:conv_index10} tracks the decay of the relative step length $\text{RSL}(x^n)$, with respect to both the number of iterations and the total CPU time. Here $E_{\min}$ represents the energy obtained by the corresponding algorithm with the stopping criterion $\text{RSL}(x^n)< 10^{-12}$.

From the iteration-wise trajectories in Figure \ref{fig:conv_index5}(a) and Figure \ref{fig:conv_index10}(a), it is evident that both pBBDCA and EPDCA require fewer iterations than the other algorithms, rapidly driving down both the objective energy and the relative variation. However, when evaluating the convergence against CPU time in Figure \ref{fig:conv_index5}(b) and Figure \ref{fig:conv_index10}(b), the practical advantage of pBBDCA becomes more pronounced. Consequently, pBBDCA not only achieves the fastest energy descent but also reaches the relative termination tolerance in the shortest time, showing its computational advantage for large-scale nonconvex optimization.

\begin{figure}[htbp!]
    \centering
    \includegraphics[width=1.0\textwidth]{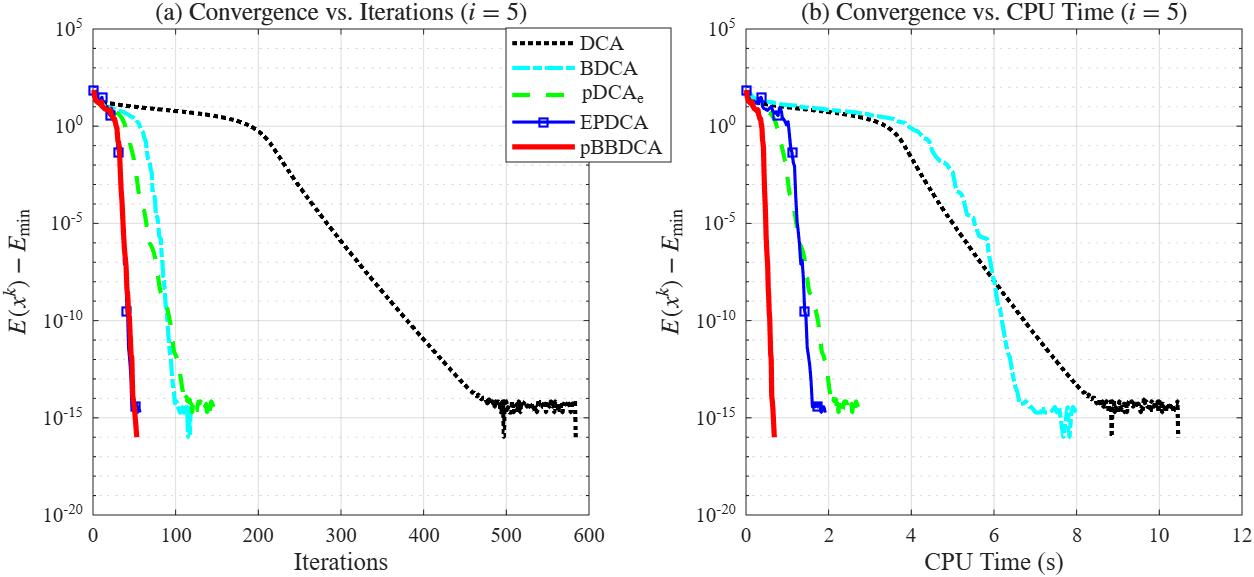}
    \caption{Convergence curves of the tested algorithms for $i=5$ ($m=3600, k=12800$), evaluated by the objective gap $E(x^k) - E_{\min}$.}
    \label{fig:conv_index5}
\end{figure}

\begin{figure}[htbp!]
    \centering
    \includegraphics[width=1.0\textwidth]{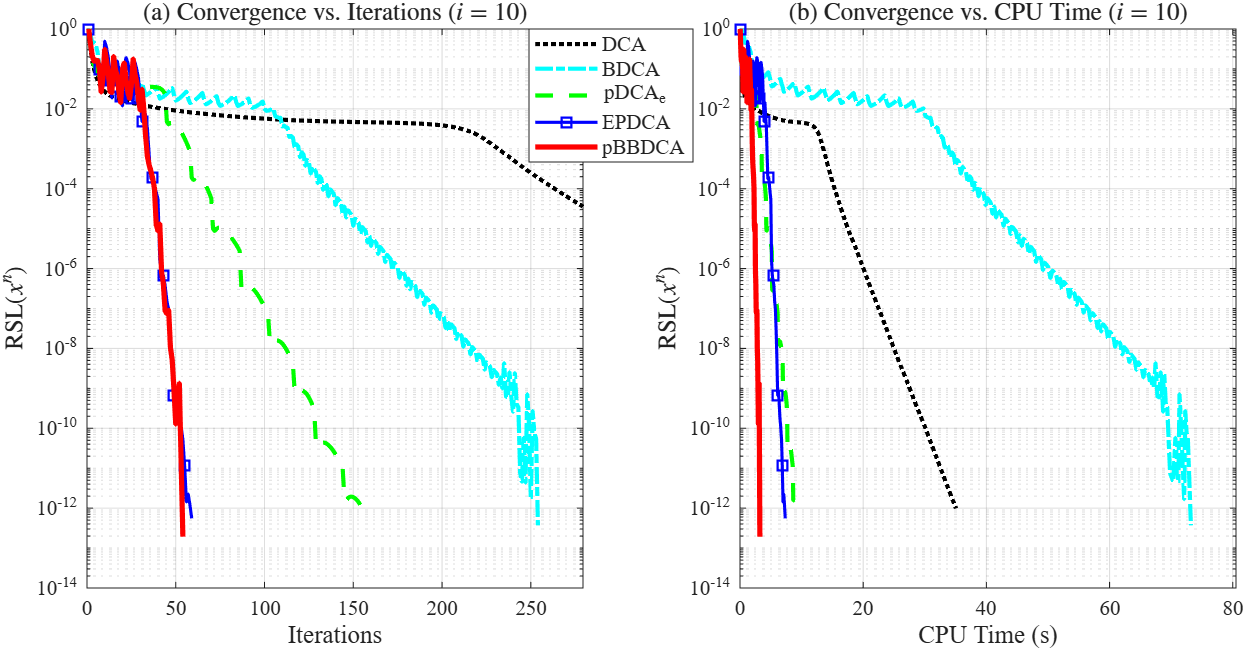} 
    \caption{Convergence curves of the tested algorithms for $i=10$ ($m=7200, k=25600$), evaluated by the relative step length $\text{RSL}(x^n)$.}
    \label{fig:conv_index10}
\end{figure}

To explain the geometric mechanism behind the acceleration, we first analyze the curvature estimation. For the smooth component $f(x) = \frac{1}{2}\|Ax-b\|^2$, the Taylor expansion determines the true second-order residual:
\begin{align*}
    \Delta f_{\text{true}}(s^{n-1}) &= f(x^n) - f(x^{n-1}) - \langle \nabla f(x^{n-1}), s^{n-1} \rangle\\ &= \frac{1}{2} \langle s^{n-1}, \nabla^2 f(x^{n-1}) s^{n-1} \rangle
    = \frac{1}{2} \|A s^{n-1}\|^2,
\end{align*}
where $s^{n-1} = x^n - x^{n-1}$. The standard DCA uses the global Lipschitz constant $L = \lambda_{\max}(A^T A)$ to bound this residual as $\Delta f_L(s^{n-1}) = \frac{L}{2} \|s^{n-1}\|^2$. In contrast, our BB step dynamically evaluates the curvature. With the gradient difference $y^{n-1} = \nabla f(x^n) - \nabla f(x^{n-1}) = A^T A s^{n-1}$, the BB1 step size computes the residual as
\begin{align*}
    \Delta f_{\text{BB}}(s^{n-1}) = \frac{1}{2\alpha_n} \|s^{n-1}\|^2 = \frac{1}{2} \left( \frac{\langle s^{n-1}, y^{n-1} \rangle}{\langle s^{n-1}, s^{n-1} \rangle} \right) \|s^{n-1}\|^2 = \frac{1}{2} \|A s^{n-1}\|^2.
\end{align*}
This formulation proves that $\Delta f_{\text{BB}} \equiv \Delta f_{\text{true}}$ exactly holds along the direction $s^{n-1}$. 

However, in the actual algorithmic execution, this historically captured curvature $\alpha_n^{-1}$ is utilized to predict the energy descent along the prospective new step $s^n = x^{n+1} - x^n$. To theoretically evaluate the accuracy of this prediction, we expand the ratio of the predicted residual to the true residual. For the standard DCA, the prediction relies on the global Lipschitz constant $L$, yielding
\begin{align*}
    \frac{\Delta f_L(s^n)}{\Delta f_{true}(s^n)} = \frac{\frac{L}{2}\|s^n\|^2}{\frac{1}{2}\langle s^n, A^T A s^n \rangle} = \frac{L}{R(s^n)},
\end{align*}
where $R(\cdot)$ denotes the Rayleigh quotient with respect to the Hessian $A^T A$. Since $L = \lambda_{\max}(A^T A)$, this ratio is at least $1$ and typically $\gg 1$ in locally flat regions, making the prediction highly conservative. In contrast, the BB1 step yields
\begin{align*}
    \frac{\Delta f_{pred}(s^n)}{\Delta f_{true}(s^n)} 
    &= \frac{\frac{1}{2\alpha_n}\|s^n\|^2}{\frac{1}{2}\langle s^n, A^T A s^n \rangle} 
    = \frac{ \left( \frac{\langle s^{n-1}, A^T A s^{n-1} \rangle}{\|s^{n-1}\|^2} \right) \|s^n\|^2 }{ \langle s^n, A^T A s^n \rangle } 
    \\
    &= \frac{ \frac{\langle s^{n-1}, A^T A s^{n-1} \rangle}{\|s^{n-1}\|^2} }{ \frac{\langle s^n, A^T A s^n \rangle}{\|s^n\|^2} } = \frac{R(s^{n-1})}{R(s^n)}.
\end{align*}
This reveals that the BB curvature estimation accuracy is dynamically governed by the Rayleigh quotient rather than a global upper bound. 

To validate this geometric behavior, Figures \ref{fig:mechanism_index5} and \ref{fig:mechanism_index10} illustrate this evaluated ratio, alongside the effective step size ratio $\tau_n / \tau_{pDCA} = \lambda_n \frac{L}{\kappa_{n-1}}$. Subplot (a) clearly demonstrates that while DCA severely overestimates the true residual, our pBBDCA tightly oscillates around the perfect match line ($y=1$). Furthermore, this precise historical curvature estimation allows pBBDCA to safely accept step sizes larger than DCA, as shown in Subplot (b). This geometric adaptivity explains the reduction in CPU time observed in Table \ref{table1}.

\begin{figure}[htbp!]
    \centering
    \includegraphics[width=1.0\textwidth]{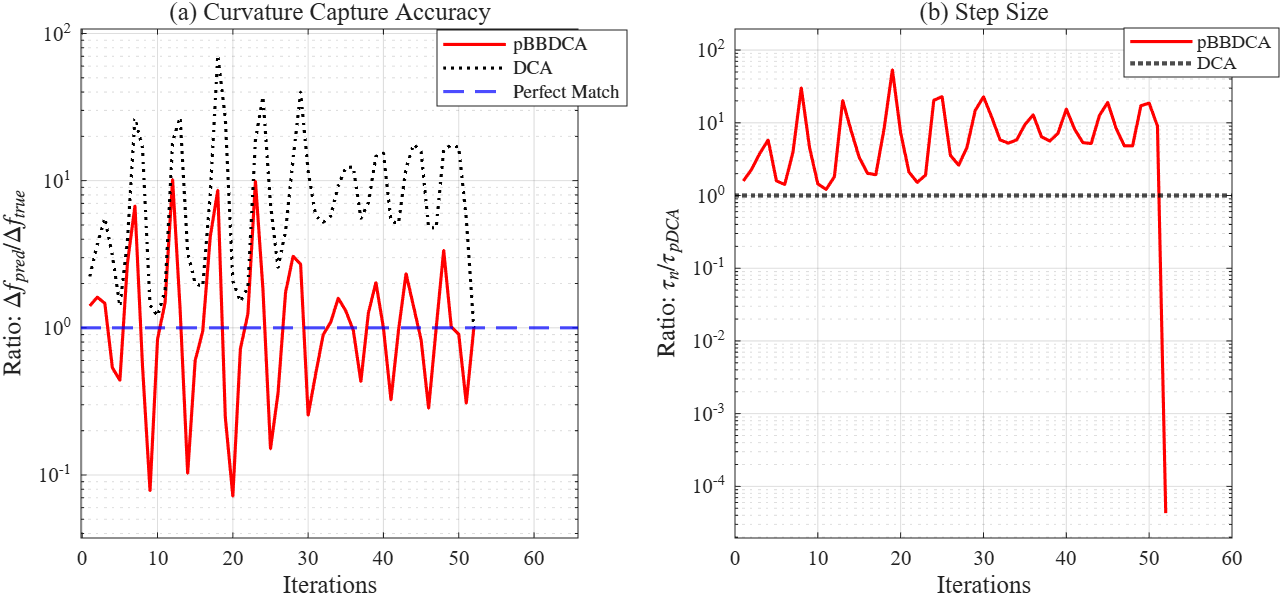} 
    \caption{Geometric mechanism analysis for $i=5$ ($m=3600, k=12800$).\\ (a) Curvature capture accuracy compared to the true residual. (b) Effective step size amplification relative to standard DCA.}
    \label{fig:mechanism_index5}
\end{figure}

\begin{figure}[htbp!]
    \centering
    \includegraphics[width=1.0\textwidth]{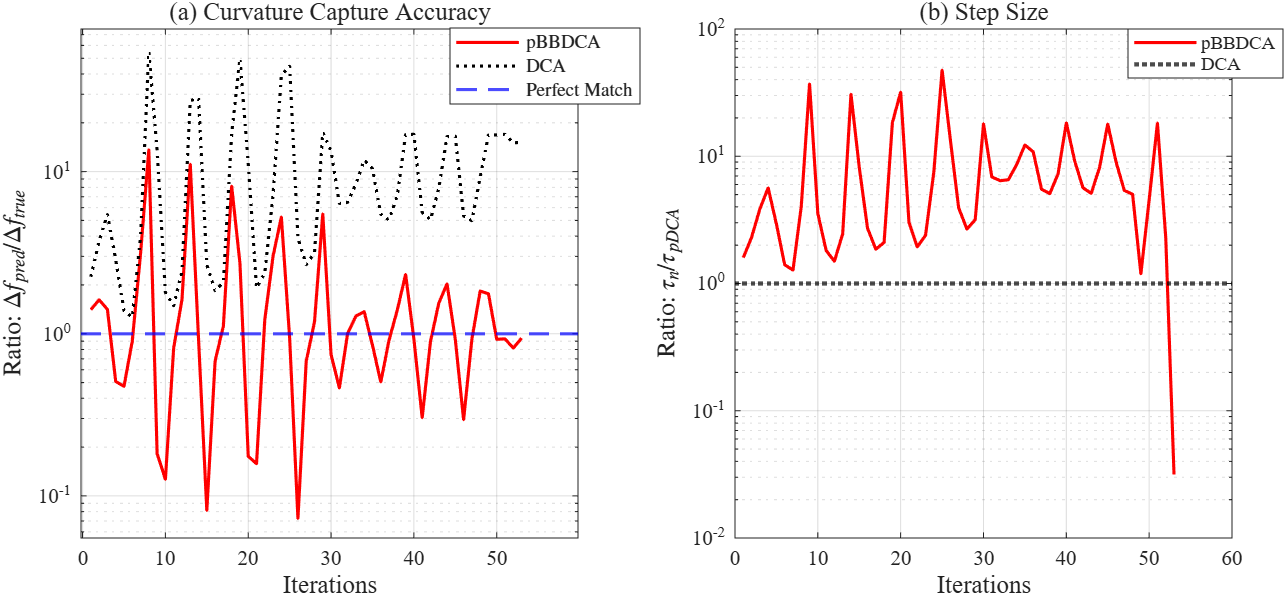} 
    \caption{Geometric mechanism analysis for $i=10$ ($m=7200, k=25600$).\\ (a) Curvature capture accuracy compared to the true residual. (b) Effective step size amplification relative to standard DCA.}
    \label{fig:mechanism_index10}
\end{figure}

We further conduct validation on binary classification tasks using LIBSVM datasets \cite{Chang2011LIBSVMAL} converted to the MATLAB format. For the classification task, we reduced the parameter $\mu$ to $5 \times 10^{-4}$. Here, experiments are initialized at \(x^{0}=0\) with stopping criteria defined as

\[
\text{RSL}(x^n)<\epsilon\text{, where }\epsilon= 10^{-i}\text{, for }i=6,7,\ldots,10.
\]

The maximum number of iterations is capped at $500000$. In Table \ref{table2}, the entry ``Max" indicates that an algorithm failed to satisfy the stopping criterion within this maximum iteration limit. Table \ref{table2} highlights the computational superiority of our proposed algorithms (pBBDCA, pABBDCA$_{\text{er}}$, and pABBDCA$_{\text{se}}$) in high-precision scenarios ($\epsilon \le 10^{-6}$). While DCA and BDCA do not meet the tighter tolerances within the iteration budget and EPDCA is hindered by repeated proximal subproblem evaluations during line search, pBBDCA achieves roughly twice the speed of EPDCA. Furthermore, the extrapolated variants deliver the best performance, and their dominance shifts depending on the precision.

\begin{table}[htbp!]
\centering
\captionsetup{justification=centering}
\caption[Solving \eqref{scad} on LIBSVM datasets]{Solving \eqref{scad} on LIBSVM datasets\protect\\ (\protect\textcircled{1}\textup{DCA} \protect\textcircled{2}\textup{BDCA} \protect\textcircled{3}\textup{pDCA}$_{\mathrm{e}}$ \protect\textcircled{4}\textup{EPDCA} \protect\textcircled{5}\textup{pBBDCA} \protect\textcircled{6}\textup{pABBDCA}$_{\mathrm{er}}$ \protect\textcircled{7}\textup{pABBDCA}$_{\mathrm{se}}$ \protect\textcircled{8}\textup{pABBDCA})}\label{table2}
\textbf{(a) Iterations}\\
\vspace{0.1cm}
{
\begin{tabular}{c|cccccccc}
\hline
\multicolumn{1}{c|}{Precision} & \multicolumn{8}{c}{Algorithms} \\ \hline
\multicolumn{1}{c|}{$\epsilon$} & \multicolumn{1}{c}{\textcircled{1}} & \multicolumn{1}{c}{\textcircled{2}} & \multicolumn{1}{c}{\textcircled{3}} & \multicolumn{1}{c}{\textcircled{4}} & \multicolumn{1}{c}{\textcircled{5}} & \multicolumn{1}{c}{\textcircled{6}} & \multicolumn{1}{c}{\textcircled{7}} & \multicolumn{1}{c}{\textcircled{8}} \\ \hline
$10^{-6}$  & 168124 & 31374 & 8402   & 3769  & 5480  & \textbf{2094} & 2220 & 8053 \\
$10^{-7}$  & Max    & Max   & 39602  & 7387  & 10621 & 2392          & \textbf{2220} & 23817 \\
$10^{-8}$  & Max    & Max   & 136802 & 11301 & 16371 & 2507          & \textbf{2220} & 260985 \\
$10^{-9}$  & Max    & Max   & 447802 & 13119 & 17642 & \textbf{3209} & 7251 & 357877 \\
$10^{-10}$ & Max    & Max   & Max    & 13528 & 17879 & \textbf{5505} & 7735 & Max \\
\hline
\end{tabular}
}

\vspace{0.4cm} 

\textbf{(b) CPU time (s)}\\
\vspace{0.1cm}
{
\begin{tabular}{c|cccccccc}
\hline
\multicolumn{1}{c|}{Precision} & \multicolumn{8}{c}{Algorithms} \\ \hline
\multicolumn{1}{c|}{$\epsilon$} & \multicolumn{1}{c}{\textcircled{1}} & \multicolumn{1}{c}{\textcircled{2}} & \multicolumn{1}{c}{\textcircled{3}} & \multicolumn{1}{c}{\textcircled{4}} & \multicolumn{1}{c}{\textcircled{5}} & \multicolumn{1}{c}{\textcircled{6}} & \multicolumn{1}{c}{\textcircled{7}} & \multicolumn{1}{c}{\textcircled{8}} \\ \hline
$10^{-6}$  & 994.6 & 850.3 & 51.9   & 66.3  & 28.7  & \textbf{23.1} & 25.2 & 45.3 \\
$10^{-7}$  & --    & --    & 299.4  & 159.6 & 79.8  & 28.4          & \textbf{25.2} & 197.3  \\
$10^{-8}$  & --    & --    & 957.0  & 207.6 & 119.1 & 31.0          & \textbf{25.2} & 1591.9 \\
$10^{-9}$  & --    & --    & 2342.4 & 248.7 & 128.6 & \textbf{37.8} & 93.3 & 2057.4 \\
$10^{-10}$ & --    & --    & --     & 259.4 & 135.2 & \textbf{65.6} & 99.5 & -- \\
\hline
\end{tabular}
}

\end{table}

Figure \ref{fig:GL} illustrates the relative step length decay against both iteration count and CPU time on logarithmic scales. The results show that our proposed algorithms (pBBDCA, pABBDCA$_{\text{er}}$, and pABBDCA$_{\text{se}}$) reach high-precision solutions faster than traditional methods and EPDCA. Besides, a direct comparison with pABBDCA demonstrates that integrating extrapolation not only further accelerates the convergence rate but also effectively dampens the severe numerical oscillations inherent to BB step sizes.

\begin{figure}[htbp!]
    \centering
    \includegraphics[width=1.0\textwidth]{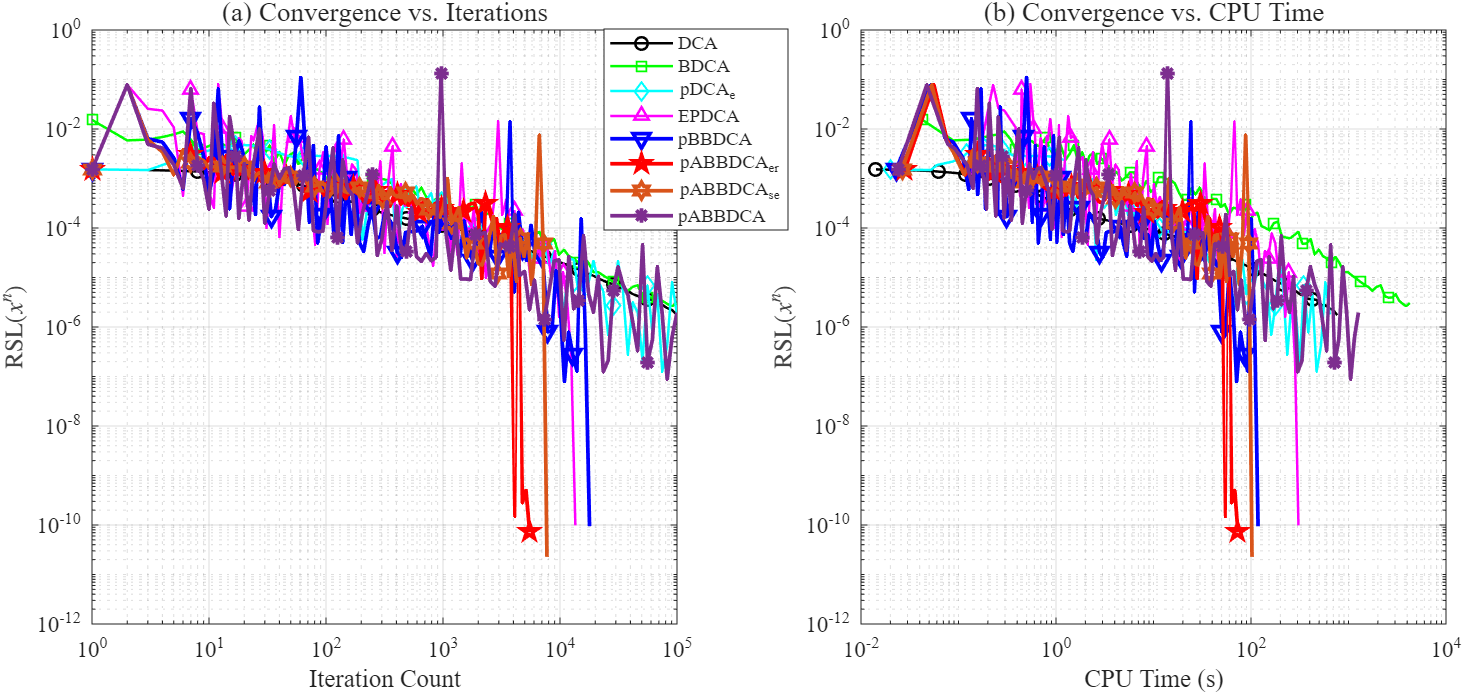}
    \caption{Convergence profiles of the tested algorithms for the LIBSVM dataset, evaluated by relative step length versus (a) Iteration Count and (b) CPU Time. Both horizontal and vertical axes are plotted on a logarithmic scale.}
    \label{fig:GL}
\end{figure}

To investigate the underlying geometric mechanism, Figure \ref{fig:extrapolation_effect} presents the angular distribution between the search direction $d^n$ and the negative smooth-gradient direction $-g^n$, where
\begin{equation*}
    g^n:=A^\top(Ax^n-b)-\nabla \widetilde P(x^n).
\end{equation*}
For both un-extrapolated baselines (pBBDCA and pABBDCA), the search directions tend to cluster near orthogonality to $g^n$, which is consistent with the numerical oscillations and slow progress observed in the late stages. In contrast, the extrapolation in pABBDCA$_{\text{er}}$ shifts the distribution toward smaller acute angles, indicating stronger alignment with the negative smooth-gradient direction. This improved geometric alignment is consistent with the faster and more stable convergence observed numerically.

\begin{figure}[htbp!]
    \centering
    \includegraphics[width=1.0\textwidth]{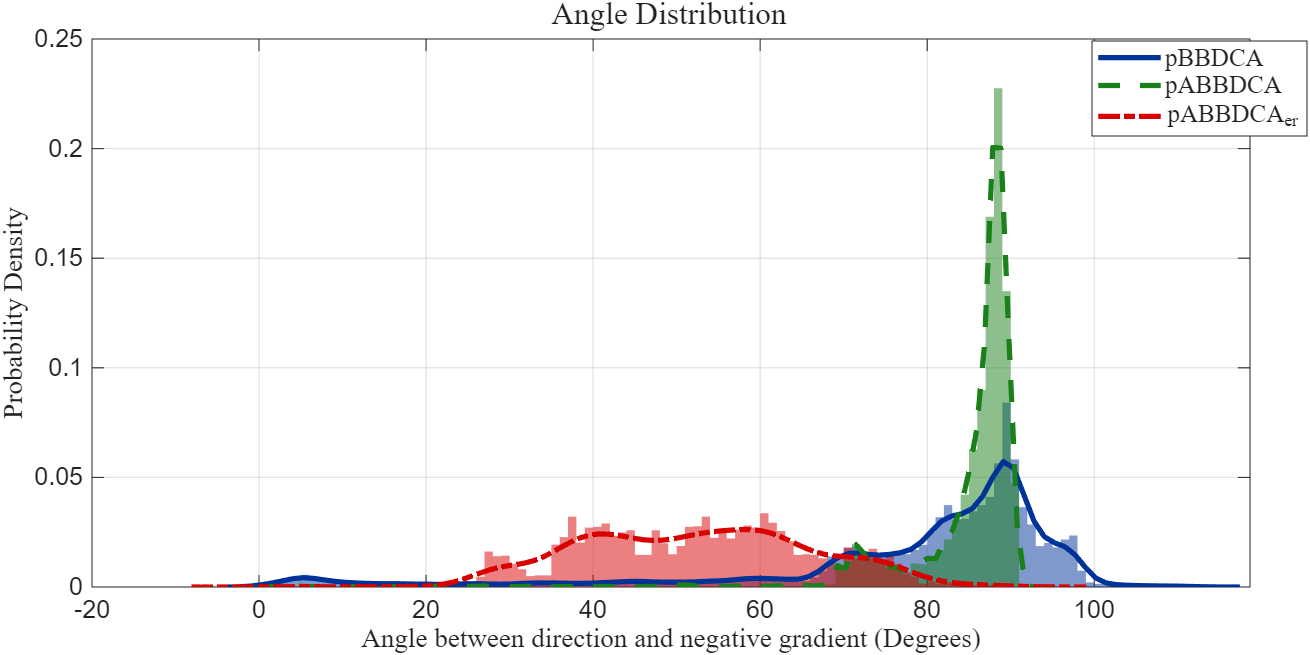}
    \caption{The impact of extrapolation on the geometric alignment of the search direction.}
    \label{fig:extrapolation_effect}
\end{figure}

\subsection{Graph Ginzburg-Landau model}
We now focus on a segmentation problem with a graph Ginzburg-Landau model, which deviates from the traditional phase-field model by integrating a prior term, thereby facilitating a semi-supervised assignment. The model is given by \cite{BF}:
\begin{equation}\label{eq:nonlocalGL}
    \min_{x\in \mathbb{R}^d}E(x) = \sum_{i,j}\frac{\epsilon}{2}w_{ij}(x(i)-x(j))^2 +\frac{1}{\epsilon}\mathbb{W}(x) + \frac{\eta}{2}\sum_i\Lambda(i)(x(i)-y(i))^2
\end{equation}
where $x$ represents the image or data, indexed by \( i \) and \( j \). The energy functional combines a double-well potential \( \mathbb{W}(x) = \frac{1}{4} \sum_{i=1}^d (x(i)^2 - 1)^2 \), a prior term $\Lambda(i)$ weighted by \( \eta \), and a nonlocal interaction term with parameter \( \epsilon \). The weights \( w_{ij} = K(i,j) \cdot N(i,j) \) are determined by feature similarity \( K(i,j) = \exp(-\|P_i - P_j\|_2^2/\sigma^2) \) and proximity \( N(i,j) \), where $P_{i}$ serves as the feature of the data point $i$ and $\sigma^2$ controls the kernel width. The matrix \( \Lambda \) and vector \( y \) encode prior knowledge. See \cite{Shensun2023} for the graph Laplacian construction. We employ Algorithm \ref{alg:prox_bb_dca} to solve the nonconvex minimization problem with the energy functional
\begin{equation*}
E(x) = f(x) + g_1(x) - g_2(x),
\end{equation*}
with
\[
 \ f(x) = \sum_{i,j}\frac{\epsilon}{2}w_{ij}(x(i)-x(j))^2 + \frac{\eta}{2}\sum_i\Lambda(i)(x(i)-y(i))^2
\]
and 
\[
g_1(x) = \frac{L}{2}\sum_i x(i)^2, \ \ g_2(x) = \frac{L}{2}\sum_i x(i)^2-\frac{1}{\epsilon}\mathbb{W}(x).
\]
 Here $L$ is the Lipschitz constant of $\frac{1}{\epsilon}\mathbb{W}(x)$. For image segmentation, we set the model parameters as $\epsilon = \eta = 10$. For our safeguarded extrapolated variant pABBDCA$_{\text{se}}$, the parameter is set to $c_0 = 0.01$.

For EPDCA and our pBBDCA, the proximal subproblems reduce to diagonal systems and are solved via direct closed-form evaluation. For all other algorithms, the subproblems contain the nonlocal graph Laplacian matrix and are solved inexactly using the Conjugate Gradient (CG) method with a termination tolerance of $10^{-5}$.

Table \ref{tab:nonlocalGLmodel} provides a comprehensive comparison of seven algorithms: the classical DCA \cite{LeThi2018}, BDCA \cite[Section 3]{Artacho2018}, the extrapolated pDCA$_\text{e}$ \cite{Wen2018}, the nonmonotone EPDCA \cite{luzhaosong}, and our proposed methods, pBBDCA, pABBDCA$_{\text{er}}$ and pABBDCA$_{\text{se}}$. The first criterion evaluated in Table \ref{tab:nonlocalGLmodel}, the DICE Bound, uses the DICE similarity coefficient to assess the quality of the segmentation. It is calculated as
\[
\text{DICE} = {2|X \cap Y|}/(|X| + |Y|),
\]
where $|X|$ and $|Y|$ represent the pixel counts of the segmentation result and the ground truth, respectively. Our DICE Bound here refers to the DICE coefficient reaching a value of 0.98, representing a high agreement between the segmentation and reference data.

\begin{table}[htbp!]
\centering
\caption[Solving \eqref{eq:nonlocalGL} on 7 different termination criteria]{Solving \eqref{eq:nonlocalGL} on 7 different termination criteria. \\(Criterion \uppercase\expandafter{\romannumeral1}: $\|\nabla E(x^n)\|$, Criterion \uppercase\expandafter{\romannumeral2}: $\|x^n - x^{n-1}\|$)\protect\\ 
\centering
(\protect
\textcircled{1}\textup{DCA} \protect\textcircled{2}\textup{BDCA} \protect\textcircled{3}\textup{pDCA}$_{\mathrm{e}}$ \protect\textcircled{4}\textup{EPDCA} \protect\textcircled{5}\textup{pBBDCA} \protect\textcircled{6}\textup{pABBDCA}$_{\mathrm{er}}$ \protect\textcircled{7}\textup{pABBDCA}$_{\mathrm{se}}$)}\label{tab:nonlocalGLmodel}
{\begin{tabular}{cccrrrrrrr}
\hline
\multicolumn{3}{c}{Criterion} & \textcircled{1} & \textcircled{2} & \textcircled{3} & \textcircled{4} & \textcircled{5} & \textcircled{6} & \textcircled{7} \\ \hline
\multicolumn{2}{c}{\multirow{2}{*}{DICE Bound}} & Iter & 11 & \textbf{8} & 36 & 69 & 116 & 15 & 15 \\ 
& & Time(s) & 15.44 & 12.89 & 5.54 & 2.63 & 2.52 &2.51 & \textbf{2.47} \\ \hline
\multirow{6}{*}{\uppercase\expandafter{\romannumeral1}} & \multirow{2}{*}{$10^{-1}$} & Iter & 35 & \textbf{23} & 105 & 180 & 154 & 35 & 35 \\
& & Time(s) & 44.26 & 32.68 & 10.10 & 7.62 & \textbf{3.35} &6.49 & 6.34 \\ 
& \multirow{2}{*}{$10^{-3}$} & Iter & 80 & 58 & 447 & 419 & 870 & 54 & \textbf{49} \\
& & Time(s) & 98.28 & 79.02 & 26.17 & 16.14 & 19.80 &9.30 & \textbf{8.37} \\ 
& \multirow{2}{*}{$10^{-5}$} & Iter & 119 & 83 & 496 & 782 & 1264 &83 & \textbf{68} \\
& & Time(s) & 145.17 & 115.46 & 27.73 & 32.09 & 28.75 &11.28 & \textbf{9.58} \\ \hline
\multirow{6}{*}{\uppercase\expandafter{\romannumeral2}} & \multirow{2}{*}{$10^{-1}$} & Iter & 37 & \textbf{36} & 109 & 53 & 68 & 41 & 40 \\
& & Time(s) & 46.65 & 50.16 & 10.18 & 2.35 & \textbf{1.47} & 7.50 & 7.21 \\ 
& \multirow{2}{*}{$10^{-3}$} & Iter & 84 & 58 & 201 & 203 & 283 &63 & \textbf{54} \\
& & Time(s) & 103.07 & 79.02 & 14.97 & 8.49 & \textbf{6.33} & 10.14 & 8.83 \\ 
& \multirow{2}{*}{$10^{-5}$} & Iter & 110 & 87 & 496 & 560 & 918 &88 & \textbf{66} \\
& & Time(s) & 134.28 & 123.10 & 27.73 & 21.06 & 20.88 & 18.05 & \textbf{9.49} \\ \hline
\end{tabular}
}
\end{table}

Table \ref{tab:nonlocalGLmodel} highlights the superior computational efficiency of our proposed framework across all precision levels. While BDCA requires fewer iterations in certain low-precision scenarios, its expensive subproblem evaluations lead to significantly slower overall running times compared to our standard pBBDCA, which consistently dominates low-to-moderate precision tasks. Crucially, under high-precision criteria where traditional algorithms suffer from severe stagnation, our extrapolated variants exhibit remarkable robustness. This confirms the synergy of alternating BB step sizes and safeguarded extrapolation when addressing complex, large-scale graph Ginzburg-Landau models.

\begin{figure}[htbp]   
  \centering            
  \subfloat[Tomato]   
  {\label{fig:tomato}\includegraphics[width=0.18\textwidth]{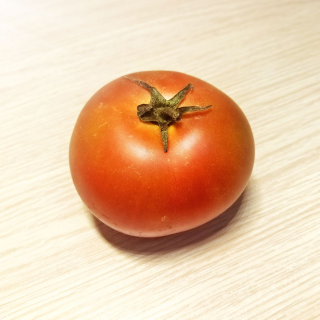}}\hfill
  \subfloat[Label]
  {\label{fig:tomatolabel}\includegraphics[width=0.18\textwidth]{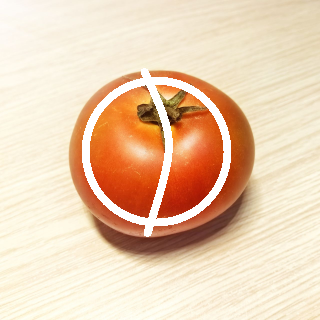}}\hfill
  \subfloat[Result 1]
  {\label{fig:result1}\includegraphics[width=0.18\textwidth]{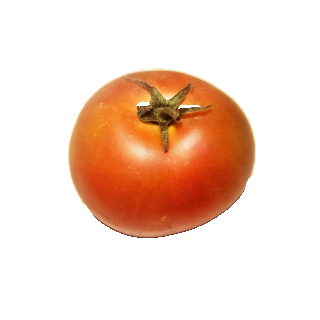}}\hfill
  \subfloat[Result 2]
  {\label{fig:result2}\includegraphics[width=0.18\textwidth]{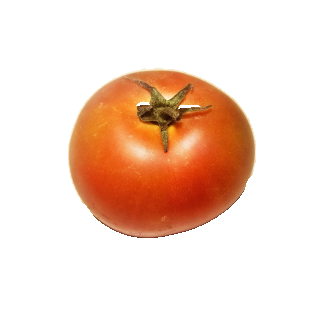}}\hfill
  \subfloat[Result 3]
  {\label{fig:result3}\includegraphics[width=0.18\textwidth]{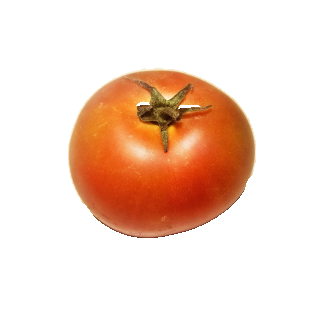}}
  \caption{The performance of segmentation assignment. (The original photograph was taken by the author.)}
  \label{fig:tomato_seg}         
\end{figure}

In Figure \ref{fig:tomato_seg}, we illustrate the segmentation results of a sample image. Figure \ref{fig:tomato} shows the original image, and Figure \ref{fig:tomatolabel} labels the only prior of the tomato to be segmented with white pixels where $\Lambda(i)=1$ (otherwise $\Lambda(i)=0$, and here we do not require the prior of the background). The last three subfigures visually present the segmentation results generated by our three proposed algorithms (pBBDCA, pABBDCA$_{\text{er}}$, and pABBDCA$_{\text{se}}$, respectively) under our ultimate termination criterion.

We finally end this section with a remark on the KL properties of the SCAD regularization~\eqref{scad} and the graph Ginzburg-Landau functional \eqref{eq:nonlocalGL}.

\begin{remark}\label{KL:two:models}
The discrete graph Ginzburg-Landau functional, being a polynomial in $x$, is semi-algebraic and consequently satisfies the Kurdyka-\L ojasiewicz (KL) property \cite[Section 2.2]{ABS}. For SCAD regularization \eqref{scad}, each scalar penalty component is a piecewise quadratic function. Following analogous reasoning to \cite[Section 5.2]{Li2018}, we establish that the energy functional $E(x)$ is a KL function. Moreover, since the class of semi-algebraic functions is closed under Fenchel conjugation and addition \cite{ABS}, the function $\mathcal{E}(x, \xi)$ in \eqref{eq:conj_potential} also satisfies the KL property. These observations verify the KL requirement for both models. 
\end{remark}

\section{Conclusion}\label{sec:conclusion}
We proposed the proximal Barzilai-Borwein DC Algorithm (pBBDCA) and its extrapolated variants to efficiently solve nonconvex and nonsmooth DC optimization problems. By decoupling the Barzilai-Borwein step size from the nonmonotone line search and employing safeguarded extrapolation mechanisms, our framework achieves robust acceleration without redundant subproblem evaluations. Global convergence is established under the Kurdyka-\L ojasiewicz property. Numerical results confirm the superior computational efficiency of the proposed methods compared to classic DC algorithms.

\noindent
{\small
\textbf{Acknowledgements}
Kelin Wu and Hongpeng Sun acknowledge the support of the National Natural Science Foundation of China under grant No. \,12271521, the National Key R\&D Program of China (2022ZD0116800),  and the Beijing Natural Science Foundation No. Z210001.
}

\section*{Data availability}

The image for segmentation and the code are available from the corresponding author on reasonable request.

\section*{Conflict of interest}
The authors declare that they have no conflict of interest.

\bibliographystyle{plain}
\bibliography{dcabb}
\end{document}